\documentclass[11pt,francais]{smfart}

\usepackage[T1]{fontenc}
\usepackage[francais,english]{babel}
\usepackage[utf8]{inputenc}
\usepackage{textcomp,multicol,enumitem}

\usepackage{amssymb,url,amsmath}
\theoremstyle{plain}\newtheorem{theo}{Th\'eor\`eme}[section]
\theoremstyle{plain}\newtheorem{prop}{Proposition}[section]
\theoremstyle{plain}\newtheorem{lem}{Lemme}[section]
\theoremstyle{plain}\newtheorem{coro}{Corollaire}[section]
\theoremstyle{plain}\newtheorem{conj}{Conjecture}[section]
\theoremstyle{plain}
\theoremstyle{definition}\newtheorem{defi}{D\'efinition}[section]
\theoremstyle{definition}\newtheorem{remarque}{Remarque}[section]
\numberwithin{equation}{section}

\newcommand{\BibTeX}{{\scshape Bib}\kern-.08em\TeX}
\newcommand{\T}{\S\kern .15em\relax }
\newcommand{\AMS}{$\mathcal{A}$\kern-.1667em\lower.5ex\hbox
	{$\mathcal{M}$}\kern-.125em$\mathcal{S}$}

\title[algébricité modulo $p$ et structures de Frobenius fortes]{algébricité modulo $p$, s\'eries hyperg\'eom\'etriques  et structures de Frobenius fortes}

\date {}
\author{Daniel Vargas-Montoya}
\address{Institut Camille Jordan, Universit\'e Claude Bernard Lyon 1, Batîment Braconnier, 21 Avenue Claude Bernard, 69100 Villeurbanne\\
}
\email{vargas@math.univ-lyon1.fr}

\keywords{Structure de Frobenius forte, r\'eduction modulo p, algébricité modulo $p$, équations différentielle $p$-adique, rigidité}
\thanks{This project has received funding from the European Research Council (ERC) under the 
	European Union's Horizon 2020 research and innovation programme 
	under the Grant Agreement No 648132. }

\begin{document}

\begin{abstract}
	Ce travail est consacré à l'étude de l'algébricité modulo $p$ des $G$-fonctions de Siegel. 
	Notre but  est de souligner la pertinence de la notion de structure de Frobenius forte, classiquement \'etudi\'ee 
	dans la théorie des équations différentielles $p$-adiques, pour l'étude d'une conjecture d'Adamczewski et Delaygue 
	concernant le degré d'algéricité de réductions modulo $p$ de $G$-fonctions. 
	Nous rendons d'abord explicite un résultat de Christol en montrant que si $f(z)$ est une $G$-fonction qui annule un opérateur différentiel dans $\mathbb{Q}(z)[d/dz]$ d'ordre $n$ qui est muni d'une structure de Frobenius forte de période $h$ pour le nombre premier $p$ et que $f(z)$ est à coefficients dans $\mathbb{Z}_{(p)}$, alors la réduction de $f$ modulo $p$ est algébrique sur $\mathbb F_p(z)$ et son degré d'algébricité est majoré par $p^{n^2h}$. En généralisant une approche introduite par Salinier, nous montrons ensuite qu'un opérateur fuchsien à coefficients dans $\mathbb{Q}(z)$, dont le groupe de monodromie est rigide et dont les exposants sont rationnels,  possède, pour presque tout nombre premier $p$, 
	une structure de Frobenius forte  de période $h$, où $h$ est majorée explicitement et 
	indépendamment  de $p$. Une version légèrement 
	différente de ce résultat a été démontré récemment par Crew en suivant une approche différente fondée sur la cohomologie $p$-adique. 
	Nous utilisons ces deux résultats pour résoudre la conjecture mentionnée  dans le cas des séries hypergéométriques généralisées. 
\end{abstract}

\maketitle

\tableofcontents
\section{Introduction}

Etant donn\'es un corps $K$ et une s\'erie formelle de plusieurs variables 
$g(z_1,\ldots,z_n)=\sum_{(i_1,\ldots,i_n)\in\mathbb{N}^n}a(i_{1},\ldots,i_{n})z_{1}^{i_{1}}\cdots z_{n}^{i_{n}}\in K[[z_1,\ldots,z_n]]$, on d\'efinit la \emph{diagonale} de $g$ comme la s\'erie formelle d'une variable
$$
f(z)=\sum_{j\geq 0}a(j,j,\ldots,j)z^j \in K[[z]].
$$ 
Lorsque $K$ est de caract\'eristique nulle, cette op\'eration est transcendante, dans le sens o\`u la diagonale d'une 
s\'erie formelle alg\'ebrique (i.e., alg\'ebrique sur le corps des fractions rationnelles $K(z_1,\ldots,z_n)$) est 
g\'en\'eralement transcendante sur le corps $K(z)$. Un exemple tr\`es simple, voir \cite{Bordiagonal},  
est donn\'e par la diagonale de 
la fraction rationnelle $\frac{4}{(2-z_1-z_2)(2-z_3-z_4)}$ qui est \'egale \`a   
$$
f(z)= \sum_{n\geq 0} \frac{1}{2^{4n}}{2n \choose n}^2z^n\in\mathbb Q[[z]].
$$
En revanche, lorsque $K$ est un corps de caract\'eristique non nulle, Furstenberg \cite{furstenberg} a montr\'e que la diagonale 
d'une s\'erie formelle rationnelle est toujours alg\'ebrique. Deligne \cite{deligne} 
a ensuite \'etendu ce r\'esultat au cas des 
diagonales de s\'eries formelles alg\'ebriques. Il souligne \'egalement  la cons\'equence 
remarquable suivante : 
si $f(z)=\sum_{n\geq 0}a(n)z^n\in \mathbb Z[[z]]$ est la diagonale d'une s\'erie formelle alg\'ebrique, 
alors pour tout nombre premier $p$, la r\'eduction modulo $p$ de $f$, c'est-\`a-dire la s\'erie formelle
$$
f_{\vert p}(z)=\sum_{n\geq 0}(a(n) \bmod p)z^n\in \mathbb F_p[[z]]\,,
$$
est alg\'ebrique sur $\mathbb F_p(z)$. 
 Un probl\`eme naturel consiste alors 
\`a \'etudier la fa\c con dont le degr\'e d'alg\'ebricit\'e de $f_{\vert p}$ varie avec $p$. 
Deligne sugg\`ere dans \cite{deligne} qu'il existe une constante $c$ ind\'ependante de $p$ 
telle que $\deg(f_{\vert p})< p^{c}$. 
Il prouve \'egalement que c'est bien le cas pour les diagonales de fonctions alg\'ebriques de deux variables. 
Le cas g\'en\'eral n'est trait\'e que plus r\'ecemment par  Adamczewski et Bell dans \cite{Bordiagonal}. Ces auteurs  montrent \'egalement, qu'on ne peut, en g\'en\'eral, esp\'erer mieux qu'une majoration polynomiale en $p$.  Lorsque $K=\overline{\mathbb Q}$,  les diagonales de s\'eries formelles alg\'ebriques forment 
une sous-classe de celle des $G$-fonctions introduite par Siegel  \cite{Siegel} en 1929.  
Rappelons que $f(z)=\sum_{n\geq0}a_nz^{n}$ est une $G$ fonction si les $a_n$ sont des nombres algébriques 
et s'il existe un nombre r\'eel $C>0$ tel que :
\begin{enumerate}[resume]
	\item la fonction $f$ annule un opérateur différentiel $L$ à coefficients dans $\overline{\mathbb{Q}}(z)$ ; 
	\item  la valeur absolue de chaque conjugué de Galois de $a_n$ est inférieure \`a $C^{n+1}$ pour tout $n\geq 0$  ;
	\item il existe une suite $D_m$ d'entiers strictement positifs tels que $D_m<C^m$ et $D_ma_n$ est un entier  algébrique pour tout $n\leq m$.
\end{enumerate}
Cette d\'efiniton implique qu'il existe un corps de nombres $K$ tel que $f(z)\in K[[z]]$. Consid\'erons un tel $K$. 
Soient $\vartheta_K$ l'anneau des entiers de $K$ et $\mathfrak{p}$ un idéal premier de $\vartheta_K$ tel que les coefficients de $f$ appartiennent à $\vartheta_{K,\mathfrak{p}}$, la localisation de $\vartheta_K$ en $\mathfrak{p}$. Notons ${\boldsymbol k}_{\mathfrak{p}}$ le corps résiduel de $\vartheta_{K,\mathfrak{p}}$, c'est-à-dire  ${\boldsymbol k}_{\mathfrak{p}}=\vartheta_{K,\mathfrak{p}}/\mathfrak{p}\vartheta_{K,\mathfrak{p}}=\vartheta_{K}/\mathfrak{p}$. 
On peut alors réduire $f$ modulo $\mathfrak{p}$ et  poser  $$f_{\mid\mathfrak{p}}(z)=\sum_{n\geq0}(a_n\bmod\mathfrak{p})z^n\in {\boldsymbol k}_{\mathfrak{p}}[[z]]$$ 
et  formuler la conjecture suivante \cite{ADconj}. 

\begin{conj}[Adamczewski--Delaygue]\label{1}
	Soient $K$ un corps de nombres et $f(z)\in K[[z]]$ une $G$-fonction. Supposons que l'ensemble 
	$\mathcal{S}$ des idéaux premiers $\mathfrak{p}$ de $\vartheta_K$ tel que $f\in\vartheta_{K,\mathfrak{p}}[[z]]$ soit infini.  Alors, on a : 
	\begin{enumerate}
		\item[{\rm (i)}] $f_{\mid\mathfrak{p}}$ est algébrique sur ${\boldsymbol k}_{\mathfrak{p}}(z)$ pour presque tout 
		$\mathfrak{p}\in\mathcal{S}$ ; 
		
		\item[{\rm (ii)}]  il existe $c>0$ tel que, pour tout $\mathfrak{p}$ v\'erifiant ${\rm (i)}$, 
		$\deg(f_{\mid\mathfrak{p}})<p^c$, o\`u $p$ d\'esigne la caract\'eristique du corps ${\boldsymbol k}_{\mathfrak{p}}$.  
	\end{enumerate}
\end{conj}

Les résultats de Deligne \cite{deligne} et d'Adamczewski et Bell \cite{Bordiagonal} mentionn\'es pr\'ec\'edemment 
montrent que cette conjecture est v\'erifi\'ee pour les diagonales 
de s\'eries alg\'ebriques et l'article \cite{Borisgfonctpu} fournit \'egalement d'autres familles d'exemples parmi les s\'eries 
hyperg\'eom\'etriques g\'en\'eralis\'ees ou les sommes multiples de produits de coefficients binomiaux.  
C'est en paticulier le cas de la série hypergéométrique $f_1(z)=_2F_1(1/2,1/2 ; 2/3,z)$ 
qui n'est pas la diagonale d'une s\'erie 
formelle alg\'ebrique car elle n'est pas globalement bornée. Par contre, d'après la proposition 8.5 de \cite{Borisgfonctpu}, en considérant l'ensemble $\mathcal{S}=\{p\in\mathcal{P} : p\equiv1\mod 3\}$,  
on obtient que pour tout $p\in\mathcal{S}$, $f_1$ peut se r\'eduire modulo $p$ et 
$f_{1\mid p}(z)=A_p(z)f_{1\mid p}(z)^p$ où $A_p(z)\in\mathbb{F}_p(z)$. 
Ainsi, $f_{1\mid p}(z)$ est algébrique sur $\mathbb{F}_p(z)$ et $\deg(f_{1\mid p})<p$.

Notons que la méthode utilisée dans \cite{Bordiagonal} est sp\'ecifique aux diagonales de s\'eries formelles alg\'ebriques, tandis que  les résultats de \cite{Borisgfonctpu} se fondent sur une analyse minutieuse de la 
valuation $p$-adique des coefficients et  ne concernent pas toutes les s\'eries hyperg\'eom\'etriques g\'en\'eralis\'ees.  
Dans cet article, notre objectif est justement de prouver une version explicite de la conjecture~\ref{1} 
pour les séries hypergéométriques g\'en\'eralis\'ees \`a param\`etres rationnels et de sortir ainsi du 
cadre des $G$-fonctions globalement born\'ees. 
Il s'agit du th\'eor\`eme~\ref{alghyp} ci-dessous. 
Rappelons que ces séries sont de la forme 
$$_{n}F_{n-1}(\underline{\alpha},\underline{\beta},z)=\sum_{j\geq0}\frac{(\alpha_1)_j\cdots(\alpha_n)_j}{(\beta_{1})_j\cdots(\beta_{n-1})_jj!}z^j$$ 
où
$\underline{\alpha}=(\alpha_{1},\ldots,\alpha_{n}), 
\underline{\beta}=(\beta_{1},\ldots,\beta_{n-1},1)\in(\mathbb{Q}\setminus\mathbb{Z}_{\leq0})^n$ et, pour $x\in\mathbb{R}$, 
$(x)_n=x(x+1)\cdots(x+n-1)$ et $(x)_0=1$. Nous désignons par $\mathbb{Z}_{(p)}$ la localisation de l'anneau $\mathbb{Z}$ en l'idéal $(p)$. Le corps résiduel de $\mathbb{Z}_{(p)}$ est alors $\mathbb{F}_p$.

\begin{theo}\label{alghyp}
	Soient $\alpha_{1},\ldots,\alpha_{n},\beta_{1},\ldots,\beta_{n-1},\beta_n=1\in\mathbb{Q}\setminus\mathbb{Z}_{\leq0}$ tels que 
	pour tout $i,j$, $\alpha_{i}-\beta_{j}\notin\mathbb{Z}$. Soit $d_{\alpha,\beta}$ le plus petit commun multiple 
	des d\'enominateurs des $\alpha_{1},\ldots,\alpha_{n}$, $\beta_{1},\ldots,\beta_{n}$ et soit $\mathcal{S}$ l'ensemble des nombres premiers $p$ tels que $p$ ne divise pas $d_{\alpha,\beta}$ et $_{n}F_{n-1}(\underline{\alpha},\underline{\beta},z)\in\mathbb{Z}_{(p)}[[z]]$.	Alors, pour tout $p\in\mathcal{S}$, la r\'eduction modulo $p$ de  $_{n}F_{n-1}(\underline{\alpha},\underline{\beta},z)$ est alg\'ebrique sur $\mathbb F_p(z)$ de degr\'e major\'e par  $p^{n^{2}\phi(d_{\alpha,\beta})}$, o\`u $\phi$ d\'esigne l'indicatrice d'Euler.
\end{theo}


La notion de \emph{structure de Frobenius forte}, introduite par Dwork \cite{DworksFf}, 
est classiquement utilis\'ee dans la th\'eorie des \'equations diff\'erentielles $p$-adiques. 
Dans cette article, nous montrons comment 
cette notion permet d'établir une stratégie générale pour attaquer la conjecture~\ref{1} (voir section~\ref{sff}). 
En particulier, la preuve du théorème~\ref{alghyp} repose sur la notion de structure de Frobenius forte. 
L'id\'ee de la d\'emonstration est la suivante. Lorsque les param\`etres $\alpha_i$ et $\beta_j$ v\'erifient les hypoth\`eses  du théorème~\ref{alghyp},  on sait que les \'equations hyperg\'eom\'etriques correspondantes ont des groupes de monodromie rigides, que leurs singularit\'es sont r\'eguli\`eres et que leurs exposants sont rationnels. 
Ces trois propri\'et\'es peuvent \^etre utilis\'ees pour d\'emontrer l'existence d'une structure de Frobenius forte 
pour presque tout $p$ dont la p\'eriode est ind\'ependante du nombre premier $p$. Dans le cas des \'equations hyperg\'eom\'etriques de Gauss, cette strat\'egie a \'et\'e mise 
en \oe uvre par  Salinier \cite{Salinier}. Notons qu'avant le travail de Salinier, Dwork   avait d\'ej\`a montr\'e par une approche diff\'erente que, sous ces hypoth\`eses,  l'op\'erateur hyperg\'eom\'etrique de Gauss est muni d'une structure de Frobenius forte pour presque tout $p$ 
(voir \cite[Chap. 7, 7.2.2]{Dworklectures}). 
Le théorème~\ref{rig} montre de fa\c con plus g\'en\'erale que les syst\`emes diff\'erentiels rigides sont munis 
d'une structure de Frobenius forte 
pour presque tout $p$ dont la p\'eriode peut \^etre major\'ee explicitement et ind\'ependamment du nombre premier $p$. 
Notre d\'emonstration de ce r\'esultat g\'en\'eralise l'approche de Salinier au cas des syst\`emes diff\'erentiels rigides ou, de fa\c con \'equivalente, au cas des \'equations diff\'erentielles sans param\`etre accessoire (cf.\ \cite{localsystems}). 
Notons par ailleurs que Crew \cite{crew} a \'egalement obtenu r\'ecemment un r\'esultat similaire en suivant une approche diff\'erente fond\'ee sur la cohomologie $p$-adique. Nous pr\'ecisons que nous n'avons pris connaissance de l'article \cite{crew} 
qu'apr\`es avoir d\'emontr\'e le th\'eor\`eme~\ref{rig}. 
Pour davantage de pr\'ecisions, notamment sur le lien entre le théorème~\ref{rig} et les r\'esultats de Katz \cite{localsystems}, Crew \cite{crew}, et Esnault et Groechenig \cite{esnault}, 
nous renvoyons le lecteur \`a la discussion pr\'ec\'edant et suivant  le th\'eor\`eme~\ref{rig}.  
D'autre part, notre théorème~\ref{algebrique}  établit un lien entre l'existence, pour un nombre premier $p$, d'une 
structure de Frobenius forte de p\'eriode $h$ pour un opérateur diff\'erentiel d'ordre $n$ et le degr\'e d'algébricité modulo $p$ des solutions (s\'eries formelles) de cet opérateur en fonction de $p$, $n$ et $h$. Il rend explicite des arguments donn\'es par Christol dans \cite{Gillesalgebriques}. 
En utilisant la rigidit\'e des \'equations hyperg\'eom\'etriques g\'en\'eralis\'ees et en combinant les th\'eor\`emes~\ref{algebrique} et \ref{rig}, on obtient finalement le th\'eor\`eme \ref{alghyp}.  Au vu de la conjecture~\ref{1} et du th\'eor\`eme~\ref{algebrique}, nous insistons sur le fait que, dans le th\'eor\`eme~\ref{rig}, il est essentiel que la p\'eriode
des structures de Frobenius fortes puisse \^etre major\'ee ind\'ependamment de $p$.

Cet article est organis\'e de la fa\c con suivante. Dans la section~\ref{sff}, nous rappelons la notion de structure 
de Frobenius forte et nous 
énon\c cons le théorème~\ref{algebrique}. Dans la section~\ref{rigidite}, nous  rappelons la notion d'opérateur 
diff\'erentiel rigide et énonçons  le théorème~\ref{rig}.  Les théorèmes~\ref{algebrique} et  \ref{rig} sont respectivement d\'emontr\'es dans les sections~\ref{algebricite} et \ref{sFfrig}. Enfin, dans la 
section~\ref{operatuerhyp}, nous  prouvons le théorème~\ref{alghyp} et l'illustrons \`a l'aide 
de quelques exemples.

\section{Structure de Frobenius forte et algébricité modulo $p$.}\label{sff}

Dans cette partie, nous rappelons la d\'efinition du corps des \'el\'ements analytiques $E_p$ et celle de 
structure de Frobenius forte d'un op\'erateur diff\'erentiel, puis nous \'enon\c cons le théorème~\ref{algebrique}.  
\'Etant donn\'e un nombre premier $p$, nous noterons $\mathbb{Z}_p$ l'anneau des entiers $p$-adiques, $\mathbb{Q}_{p}$ 
le corps des nombres $p$-adiques et $\mathbb{C}_{p}$ le compl\'et\'e de la cl\^oture alg\'ebrique de $\mathbb{Q}_{p}$. Nous rappelons que la valuation de $\mathbb{Z}_{p}$ s'\'etend de mani\`ere unique \`a $\mathbb{C}_{p}$. Nous d\'esignerons par $\pi_{p}$, 
un \'el\'ement de $\mathbb{C}_{p}$ v\'erifiant $\pi_{p}^{p-1}=-p$. 

\subsection{\'El\'ements analytiques}
Soit $K$ un corps ultram\'etrique de caract\'eristique nulle muni de la valuation $\vert \cdot \vert$ et 
$k$ son corps r\'esiduel de caract\'eristique $p$. Pour $\vert x\vert\leq1$ nous notons $\bar{x}$ l'\'el\'ement de $k$ qui repr\'esente la classe r\'esiduelle de $x$. Nous supposons que $K$ est complet pour la norme $\vert \cdot \vert$. 
Nous dirons que $\sigma:K\rightarrow K$ est un automorphisme de Frobenius si les conditions suivantes sont v\'erifi\'ees :
\begin{enumerate}
	\item pour tout $x\in K$, $|\sigma(x)|=|x|$ ;
	\item pour tout $x\in K$, $|x|\leq 1$, $|\sigma(x)-x^{p}|<1$. Autrement dit, $\overline{\sigma}:k\rightarrow k$ donn\'e par $\overline{\sigma}(\overline{x})=\overline{\sigma(x)}$ est l'endomorphisme de Frobenius. Remarquons que $\overline{\sigma}$ est bien d\'efini par le point 1.
\end{enumerate} 

\begin{remarque}\label{frobcp}
	La proposition 1.10.1 de \cite{Gillesmoduldiff} montre que le corps $\mathbb{C}_{p}$ poss\`ede un automorphisme de Frobenius, 
	mais celui-ci  n'est pas unique. Dans la suite, nous fixons un tel automorphisme que nous notons $Frob:\mathbb{C}_{p}\rightarrow\mathbb{C}_{p}$ et que nous appellerons l'automorphisme de Frobenius de $\mathbb{C}_{p}$.
\end{remarque}

Nous montrons \`a pr\'esent comment construire le corps des \'el\'ements analytiques.  
Nous d\'esignons par $\mathcal{W}$ l'anneau d'Amice qui est 
l'ensemble des s\'eries formelles \[f(z)=\sum_{n\in\mathbb{Z}}a_nz^{n}\] 
telles que les $a_{n}$ sont des \'el\'ements de $K$ dont la valeur absolue est born\'ee et  tend vers z\'ero lorsque $n$ 
tend n\'egativement vers l'\'infini. 
D'apr\`es la proposition 1.1 de \cite{Gillesalgebriques}, l'anneau $\mathcal{W}$ est un $K$-espace vectoriel complet pour la norme d\'efinie par $|f|=sup\{|a_{n}| : n\in\mathbb Z\}$. 
L'anneau $K[z]$ est contenu dans l'anneau $\mathcal{W}$ et ainsi, l'anneau $K[z]$ est muni de la norme de Gauss 
\[\left|\sum a_{j}z^{j}\right|_{\mathcal G}={sup |a_{j}|}.\] 
De plus, d'apr\`es la proposition 1.2 de \cite{Gillesalgebriques}, tout \'el\'ement non nul de $K[z]$ est inversible dans $\mathcal{W}$, ce qui implique que 
l'anneau des fractions rationnelles $K(z)$ est contenu dans $\mathcal{W}$. La norme de $\mathcal{W}$ 
induit une norme sur $K(z)$ qui s'exprime comme  
\[\left|\frac{\sum a_{j}z^{j}}{\sum b_{i}z^{i}}\right|_{\mathcal G}=\frac{sup |a_{j}|}{sup |b_{i}|}.\] 
Comme $\mathcal{W}$ est complet, le compl\'et\'e de $K(z)$ pour la norme de Gauss est \'egalement contenu dans 
$\mathcal{W}$. 

\begin{defi}[\'El\'ements analytiques]  
	Le corps des \'el\'ements analytiques $E_K$ est le compl\'et\'e du corps $K(z)$ pour la norme de Gauss dans 
	$\mathcal{W}$. Dans le cas o\`u $K=\mathbb{C}_{p}$, le corps des \'el\'ements analytiques est not\'e $E_{p}$ 
	et l'analogue de $\mathcal{W}$ est not\'e $\mathcal{W}_{p}$.
\end{defi}

\begin{remarque}
	Pour tout $f\in K(z)$, on a \[\left|\frac{d}{dz}(f)\right|_{\mathcal G}\leq|f|_{\mathcal G}.\] Ainsi, la d\'eriviation $\frac{d}{dz}:K(z)\rightarrow K(z)$ est une fonction continue pour la norme de Gauss et elle s'\'etend naturellement au corps $E_K$ des \'el\'ements analytiques. On note encore son extension $\frac{d}{dz}:E_K\rightarrow E_K$.
\end{remarque}

\begin{defi}[$E_p$-\'equivalence]
	Soient $A$ et $B$ dans ${\mathcal M}_n(E_{p})$. Nous disons que $A$ et $B$ sont $E_{p}$-\'equivalentes s'il existe $H\in {\rm GL}_n(E_{p})$ telle que \[\frac{d}{dz}H=AH-HB.\]
\end{defi}

\subsection{Structure de Frobenius forte}

Nous d\'efinissons l'application $Frob_{z^{p}}:\mathbb{C}_{p}(z)\rightarrow E_{p}$ par $$
Frob_{z^p}\left(\frac{\sum a_{i}z^i}{\sum b_{j}z^j}\right)=\frac{\sum Frob(a_{i})z^{pi}}{\sum Frob(b_{j})z^{pj}}.$$ Cette application est une isom\'etrie. Il s'agit donc d'une application continue qui s'\'etend au corps des \'el\'ements analytiques $E_{p}$. Nous la notons encore $Frob_{z^p}:E_{p}\rightarrow E_{p}$. C'est \`a nouveau une isom\'etrie et, pour tout $e\in E_{p}$, on a \begin{equation}\label{derive}
\frac{d}{dz}(Frob_{z^p}(e))=pz^{p-1}\left(Frob_{z^p}(\frac{d}{dz}e)\right).
\end{equation}
Soit $A$ une matrice de taille $n$ \`a coefficients dans $E_{p}$, nous consid\'erons la matrice $F_{z^p}(A):=\frac{d}{dz}(z^p)A^{Frob_{z^p}}=pz^{p-1}A^{Frob_{z^p}}$, o\`u $A^{Frob_{z^p}}$ est la matrice obtenue apr\`es avoir appliqu\'e $Frob_{z^p}$ \`a chaque entr\'ee de $A$.

Étant donné un opérateur différentiel d'ordre $n$
\begin{equation}\label{L}
L:=a_0(z)\frac{d}{dz^n}+a_{1}(z)\frac{d}{dz^{n-1}}+\cdots+a_{n-1}(z)\frac{d}{dz}+a_n(z)\in\mathbb{Q}[z][d/dz]\,,
\end{equation}
on d\'efinit la \emph{matrice compagnon} associ\'ee \`a $L$ par	\[
A=\begin{pmatrix}
0 & 1 & 0 & \dots & 0 & 0\\
0 & 0 & 1 & \dots & 0 & 0\\
\vdots & \vdots & \vdots & \vdots & \vdots & \vdots \\
0 & 0 & 0 & \ldots & 0 & 1\\
\frac{-a_{n}(z)}{a_0(z)} & \frac{-a_{n-1}(z)}{a_0(z)} & \frac{-a_{n-2}(z)}{a_0(z)} & \ldots & \frac{-a_{2}(z)}{a_0(z)} & \frac{-a_{1}(z)}{a_0(z)}\\
\end{pmatrix}.
\] 

\begin{defi}[Structure de Frobenius forte]
	Soit  $L\in\mathbb{Q}(z)[d/dz]$ un op\'erateur diff\'erentiel et soit $A$ sa matrice compagnon. Nous disons que $L$ a une structure de Frobenius forte de période h pour un nombre premier $p$  s'il existe un entier strictement positif  $h$ tel que la matrice $A$ et la matrice obtenue en appliquant $h$-fois $F_{z^p}$ \`a $A$ sont $E_{p}$-\'equivalentes. Comme $A\in {\mathcal M}_{n}(\mathbb{Q}(z))$, cela revient \`a dire qu'il existe $H\in {\rm GL}_{n}(E_{p})$ telle que 
	\begin{equation}\label{isofrob1}
	\frac{d}{dz} H=AH-p^{h}z^{p^{h}-1}HA(z^{p^{h}}).
	\end{equation}
	Le plus petit entier $h\geq 1$  ayant cette propriété est appel\'e \emph{la p\'eriode} de la structure de Frobenius associ\'ee au couple $(L,p)$.
\end{defi}

Par exemple, d'après \cite[Chap. 22, Theorem 22.2.1]{Kedlaya}\footnote{Voir aussi la discussion p. 351 de \cite{Kedlaya} et les r\'ef\'erences associ\'ees.} ou \cite[Chap. V, p. 111 ]{andre}, les \emph{équations de Picard-Fuchs} sont munies d'une structure de Frobenius forte pour presque tout nombre premier $p$. 
Le résultat qui suit montre le lien entre l'existence d'une structure de Frobenius forte pour $p$ de l'opérateur différentiel $L$ et l'algébricité modulo $p$ de ses solutions.

\begin{theo}\label{algebrique}
	Soit $L\in\mathbb{Q}(z)[d/dz]$ d'ordre $n$. 
	Soit $p$ un nombre premier pour lequel l'op\'erateur diff\'erentiel $L$ a une structure de Frobenius forte de p\'eriode $h$ et  $f(z)=\sum_{n\geq0}a(n)z^{n}\in\mathbb{Z}_{(p)}[[z]]$ une solution de $L$. 
	Soit $f_{\mid p}$ la r\'eduction de $f$ modulo $p\mathbb{Z}_{(p)}$. Alors la s\'erie formelle $f_{\mid p}$ est une s\'erie alg\'ebrique sur 
	$\mathbb{F}_{p}(z)$ de  degr\'e major\'e par $p^{n^{2}h}$.
\end{theo}

Notons que nous pourrions aussi énoncer ce théorème dans le cas d'un op\'erateur \`a coefficients dans 
$\overline{\mathbb Q}[z]$, la démonstration s'obtiendrait de la m\^eme fa\c con. 
Le théorème~\ref{algebrique} nous dit que pour montrer le point 1 de la conjecture~\ref{1} il suffit de voir que $f$ annule un opérateur muni d'une structure de Frobenius forte pour presque tout $p$ dans $\mathcal S$.

\begin{remarque}\label{rem:gfonction}
	Si $L\in\mathbb{Q}[z][d/dz]$ est un op\'erateur diff\'erentiel muni d'une structure de Frobenius forte pour presque tout $p$ et $f\in\mathbb{Q}[[z]]$ est une solution de $L$, alors  $f$ est une $G$ fonction. En effet, l'hypoth\`ese faite sur $L$ 
	implique que le rayon de convergence au point générique est égal à 1 
	pour presque tout $p$. Cela d\'ecoule des 
	propositions 4.1.2, 4.6.4 et 4.7.2 de \cite{Gillesmoduldiff}. D'après la proposition 5.1 et le théorème 6.1 de \cite[Chap. III]{Dworkgfunciones}, on obtient alors que les singularit\'es de $L$ sont r\'eguli\`eres \`a exposants rationnels. Enfin, en combinant cette propri\'et\'e avec le théorème 4.2 \cite[Chap. VII]{Dworkgfunciones} et la proposition 1.1 de \cite[Chap. VIII]{Dworkgfunciones}, on obtient que $f$ est une $G$ fonction.
\end{remarque}

\section{Structure de Frobenius forte et rigidité}\label{rigidite}

Notre troisième résultat, le théorème~\ref{rig}, concerne  les opérateurs différentiels à coefficients dans $\mathbb{Q}[z]$ dont le groupe de monodromie est rigide.
 Nous commençons par rappeler la notion de système différentiel rigide. Étant donné le système différentiel \begin{equation}\label{A}
	\frac{d}{dz}y=Ay,\quad A\in {\mathcal M}_{n}(\mathbb{C}(z)),
	\end{equation}
nous disons que $\gamma\in\mathbb{C}$ est un point \emph{singulier} du système différentiel~\eqref{A} si $\gamma$ est un pôle de $A$. L'infini est un point singulier de~\eqref{A} si zéro est un point singulier du système différentiel obtenu après avoir appliqué le changement de variable $z\mapsto1/z$ au système~\eqref{A}. Nous disons que $\gamma\in\mathbb{C}$ est un point \emph{singulier régulier} du système différentiel~\eqref{A} si $\gamma$ est un point singulier et s'il existe une matrice $A_{\gamma}\in {\mathcal M}_{n}(\mathbb{C}(z))$ telle que la matrice $(z-\gamma)A_{\gamma}$ n'a pas de pôle en $\gamma$ et qu'il existe une matrice $P\in {\rm GL}_n(\mathbb{C}(\{z-\gamma\}))$ telle que $\frac{d}{dz}P=AP-PA_{\gamma}$, où $\mathbb{C}(\{z-\gamma\})$ est le corps des séries de Laurent convergentes au voisinage de $\gamma$. L'infini est un point singulier régulier de~\eqref{A} si zéro est un point singulier régulier du système différentiel obtenu après avoir appliqué le changement de variable $z\mapsto1/z$ au système~\eqref{A}. Le système différentiel~\eqref{A} est \emph{fuchsien} si tous ses points singuliers sont singuliers réguliers. 
 Soit $x$ un point non singulier du système différentiel~\eqref{A}, d'apr\'es la th\'eorie classique de Cauchy $$Sol(A)_{x}=\{y\in\mathbb{C}(\{z-x\})^n \text{ et } \frac{d}{dz}y=Ay\}$$
est un $\mathbb{C}$-espace vectoriel de dimension $n$. 
Soient $F=\{y_{1},\dots, y_{n}\}$ une base de $Sol(A)_x$, $S$ l'ensemble de points singuliers de~\eqref{A} et $\gamma\in S$. Les coordonnées des vecteurs $y_{1},\ldots, y_{n}$ peuvent \^etre prolong\'ees analytiquement le long du lacet $\tau$, où $[\tau]\in\Pi_{1}(\mathbb{C}\setminus S,x)$ et $\tau$ est un lacet autour de $\gamma$ tel que le groupe engendré par $[\tau]$ est $\mathbb{Z}$. Ainsi, on obtient un nouvel ensemble $\widetilde{F}$ qui sera encore une base de $Sol(A)_{x}$. Alors la matrice de \emph{monodromie locale en $\gamma$}, not\'ee $M(A,\gamma)\in {\rm GL}_n(\mathbb{C})$, est la matrice de changement de base de $F$ vers $\widetilde{F}$. Le théorème de monodromie complexe nous garanti que si $[\tau]=[\alpha]$,  on obtient  encore l'ensemble $\widetilde{F}$ en prolongeant les coordonnées des vecteurs $y_1,\ldots, y_n$ le long du  lacet $\alpha$. Écrivons $S=\{\gamma_{1},\ldots,\gamma_{r}\}\subset\mathbb{C}\cup\{\infty\}$. Le \emph{groupe de monodromie} de~\eqref{A} est le groupe engendr\'e par les matrices $M(A,\gamma_{1}),\ldots, M(A,\gamma_{r})$ qui satisfont \`a la relation $M(A,\gamma_{1})\cdots M(A,\gamma_{r})=Id$. D'après la construction, le groupe de monodromie dépend de la base $F$. Si nous prenons une autre base $F_1$ de $Sol(A)_x$ alors ces deux groupes de monodromie sont conjugués. Ainsi le groupe de monodromie est unique à conjugaison près.

\medskip

Soit \begin{equation}\label{A'}
\frac{d}{dz}y=A'y, \quad A'\in {\mathcal M}_{n}(\mathbb{C}(z))
\end{equation} un système différentiel. Supposons que les systèmes~\eqref{A} et \eqref{A'} sont fuchsiens et que z\'ero est un point singulier régulier de \eqref{A} et de \eqref{A'}. On dit que $A$ et $A'$ sont \emph{localement \'equivalentes en z\'ero} si les matrices de monodromie locale en z\'ero de \eqref{A} et \eqref{A'} sont conjugu\'ees. D'apr\`es le th\'eor\`eme 5.1 de \cite{Singer}, cela revient \`a dire qu'il existe $P$ une matrice inversible \`a coefficients dans $\mathbb{C}(\{z\})$ telle que \[\frac{d}{dz}P=AP-PA',\]
o\`u $\mathbb{C}(\{z\})$ est le corps des s\'eries de Laurent qui convergent.
Soit $x_{0}$ un point singulier r\'egulier de \eqref{A} et de \eqref{A'} On dit que $A$ et $A'$ sont \emph{localement \'equivalentes en} $x_{0}$ si, apr\`es application du changement de variable $z\mapsto z-x_{0}$ \`a \eqref{A} et \eqref{A'}, les nouveaux systèmes sont localement \'equivalents en z\'ero.
On dit que $A$ et $A'$ sont \emph{localement \'equivalentes}, si elles sont localement \'equivalentes en tous points.

\begin{defi}[Système rigide]
	Soient $A\in {\mathcal M}_{n}(\mathbb{C}(z))$ et $\frac{d}{dz}y=Ay$ un syst\`eme diff\'erentiel fuchsien. 
	Un tel système  est dit \emph{rigide} si, pour tout système $\frac{d}{dz}y=A'y$ fuchsien tel que $A$ et $A'$ sont localement \'equivalentes, il existe $P\in {\rm GL}_n(\mathbb{C}(z))$ telle que \[\frac{d}{dz}P=AP-PA'.\] 
	Autrement dit, les matrices $A$ et $A'$ sont \emph{$\mathbb{C}(z)$-\'equivalentes}. 
	
	En général pour un corps quelconque $K$, nous disons que deux matrices $A,A'\in {\mathcal M}_{n}(K(z))$ sont \emph{$K(z)$-équivalentes} s'il existe $P\in {\rm GL}_n(K(z))$ telle que $\frac{d}{dz}P=AP-PA'$.
\end{defi}
\begin{defi}[Groupe rigide] 
	Soient $g_{1},\ldots,g_{r}\in {\rm GL}_n(\mathbb{C})$ et $G$ le groupe engendr\'e par ces matrices. 
	Le r-uplet $g_{1},\ldots,g_{r}$  est dit irr\'eductible si $G$ agit de mani\`ere irr\'eductible sur $\mathbb{C}^{n}$. 
	Le groupe $G$ est dit rigide si les conditions suivantes sont v\'erifi\'ees :
	\begin{enumerate}
		\item le r-uplet $g_{1},\ldots,g_{r}$ est irr\'eductible ;
		\item on a $g_{1}\cdots g_{r}=Id$ ;
		\item pour tout r-uplet $\widetilde{g_{1}},\ldots,\widetilde{g_{r}}$ tel que $\widetilde{g_{1}}\cdots\widetilde{g_{r}}=Id$, 
		o\`u $\widetilde{g_{i}}$ est conjugu\'e \`a $g_{i}$, il existe une matrice $U\in {\rm GL}_n(\mathbb{C})$ telle que 
		$\widetilde{g_{i}}=Ug_{i}U^{-1}$ pour tout $i\in\{1,\ldots,r\}.$
	\end{enumerate}	
\end{defi}

\begin{prop}\label{rigide}
	Soient $A\in {\mathcal M}_{n}(\mathbb{C}(z))$ et $\frac{d}{dz}y=Ay$ fuchsien. Si le groupe de monodromie de $\frac{d}{dz}y=Ay$ est rigide alors le système $\frac{d}{dz}y=Ay$ est rigide. 
\end{prop}
En effet, cette proposition d\'ecoule de la proposition suivante qui est un cas particulier du th\'eor\`eme 6.15 de \cite{Singer}.
\begin{prop}\label{1-1}
	Soient $A,A'\in {\mathcal M}_{n}(\mathbb{C}(z))$. Si $\frac{d}{dz}y=Ay$ et $\frac{d}{dz}y=A'y$ sont deux systèmes fuchsiens dont les groupes de monodromie sont conjugu\'es, alors les matrices $A$ et $A'$ sont $\mathbb{C}(z)$-équivalentes.
\end{prop}

Soient $L\in\mathbb{C}(z)[d/dz]$ et $A$ sa matrice compagnon. Si $L$ est un opérateur fuchsien, alors le système différentiel $\frac{d}{dz}y=Ay$ est fuchsien (cf.\ remarque~\ref{L et A}). 

\begin{defi}[Opérateur rigide]
	Soient $L\in\mathbb{C}(z)[d/dz]$ fuchsien et $A$ sa matrice compagnon. L'opérateur $L$ est rigide si le système différentiel $\frac{d}{dz}y=Ay$ est dit rigide.
\end{defi}

Le calcul de la matrice de monodromie locale n'est pas facile. Par contre, dans le cas où $\gamma$ est un point singulier régulier nous pouvons la calculer à conjugaison près. Nous ferons cela dans le lemme suivant. Avant, nous rappelons la notion \emph{d'exposants en $\gamma$}. Soit $\gamma\in\mathbb{C}$ un point singulier régulier du système~\eqref{A}. Cela veut dire qu'il existe une matrice $A_{\gamma}\in\ {\mathcal M}_{n}(\mathbb{C}(z))$ telle que les matrices $A$ et $A_{\gamma}$ sont localement équivalentes en $\gamma$ et la matrice $(z-\gamma)A_{\gamma}$ n'a pas de pôle en $\gamma$. Les exposants en $\gamma$ sont les valeurs propres de la matrice $[(z-\gamma)A_{\gamma}](\gamma)$. Dans le cas que $\gamma=\infty$, $\gamma$ est un point singulier régulier s'il existe une matrice $A_{\infty}\in {\mathcal M}_{n}(\mathbb{C}(z))$ telle que les matrices $A$ et $A_{\infty}$ sont localement équivalentes en $\gamma$ et la matrice $zA_{\infty}$ n'a pas de pôle en $\gamma$. Les exposants en l'infini sont les valeurs propres de la matrice $[zA_{\infty}](\gamma)$. Notons que cette définition dépend de la matrice $A_{\gamma}$. Par contre, d'après le lemme 2.4 de \cite[Chap V]{Dworkgfunciones}, si $A'_{\gamma}$ est telle que $A$ et $A'_{\gamma}$ sont localement équivalentes en $\gamma$ et $(z-\gamma)A'_{\gamma}$ n'a pas de pôle en $\gamma$, alors les valeurs propres de $[(z-\gamma)A_{\gamma}](\gamma)$ et les valeurs propres de $[(z-\gamma)A'_{\gamma}](\gamma)$ sont égales modulo $\mathbb{Z}$. 

\begin{lem}\label{monodromie}
	Soient $A\in {\mathcal M}_{n}(\mathbb{C}(z))$ et $\gamma\in\mathbb{C}\cup\{\infty\}$ un point singulier régulier du système différentiel $\frac{d}{dz}y=Ay$. Alors il existe une matrice $C\in {\mathcal M}_{n}(\mathbb{C})$ telle que $\exp(2\pi iC)$ est conjuguée à la matrice de monodromie locale de $A$ en $\gamma$ et satisfaisant aux  
	 deux propriétés suivantes:
	
	{\rm \textbf{(a)}} si $\lambda,\beta$ sont deux valeurs propres diff\'erentes de $C$, alors $\lambda-\beta\notin\mathbb{Z}$,
	
	{\rm \textbf{(b)}} l'ensemble des exposants de $A$ en $\gamma$ et l'ensemble des valeurs propres de $C$ sont égaux modulo $\mathbb{Z}$.
	
\end{lem}

\begin{proof}
	Sans perdre de généralité supposons que $\gamma=0$. Soit $A_{0}(z)\in {\mathcal M}_{n}(\mathbb{C}(z))$ telle que les matrices $A(z)$ et $\frac{1}{z}A_{0}(z)$ sont localement équivalentes en zéro et $A_0$ n'a pas de pôle en zéro. Comme $A_0(z)\in {\mathcal M}_{n}(\mathbb{C}[[z]])$, le lemme 8.2 et le corollaire 8.3 de \cite[chap. III]{Dworkgfunciones} nous assure l'existence d'une matrice $W_{0}(z)\in {\mathcal M}_n(\mathbb{C}[[z]])$ telle que les matrices $\frac{1}{z}A_0(z)$ et $\frac{1}{z}W_0(z)$ sont localement équivalentes en zéro et les valeurs propres de $C:=W_{0}(0)$ satisfont aux conditions \textbf{ (a)} et \textbf{(b)} de l'\'enonc\'e. D'apr\`es le th\'eor\`eme 5.1 de \cite{Singer}, on obtient que la matrice de monodromie locale du syst\`eme $\frac{d}{dz}y=\frac{1}{z}W_{0}(z)y$ en $0$ est conjugu\'ee \`a $\exp(2\pi i C)$. Puisque $\frac{1}{z}A_0(z)$ et $\frac{1}{z}W_0(z)$ sont localement équivalentes en zéro, alors toujours par le th\'eor\`eme 5.1 de \cite{Singer}, la matrice de monodromie locale en zéro du syst\`eme $\frac{d}{dz}y=\frac{1}{z}A_{0}(z)y$ est conjugu\'ee \`a $\exp(2\pi i C)$. Finalement, comme $A(z)$ et $\frac{1}{z}A_0(z)$ sont localement équivalentes en zéro alors, d'après le théorème 5.1 de \cite{Singer}, la matrice de monodromie locale en zéro de $\frac{d}{dz}y=A(z)y$ est conjugu\'ee \`a $\exp(2\pi i C)$ et comme nous l'avons d\'ej\`a \'ecrit, $C$ satisfait aux conditions \textbf{(a)} et \textbf{(b)}.	
\end{proof}
\begin{remarque}\label{equivalence}
	Soit $K$ un corps algébriquement clos. Il suit de la démonstration du lemme 8.2 et du corollaire 8.3 de \cite[Chap.III]{Dworkgfunciones} que, si $A_0\in {\mathcal M}_{n}(K(z))$, alors $W_0\in {\mathcal M}_{n}(K(z))$ et les matrices $\frac{1}{z}A_0(z)$ et $\frac{1}{z}W_0(z)$ sont $K(z)$-équivalentes. De plus, la matrice $C:=W_0(0)$ v\'erifie les conditions du lemme~\ref{monodromie}.
\end{remarque}

 D'apr\`es Katz  \cite[Theorem 9.4]{localsystems}, si $L$ est un op\'erateur rigide alors le module différentiel défini par $L$ est un sous-module différentiel d'un module différentiel associ\'e \`a une \'equation de \emph{Picard-Fuchs}.  Or les \'equations de Picard-Fuchs sont munies  d'une structure de Frobenius forte pour presque tout $p$ (voir par exemple \cite[Theorem 22.2.1]{Kedlaya}). Ainsi, on peut s'attendre \`a ce qu'un op\'erateur rigide  soit muni d'une structure de Frobenius forte pour presque tout $p$. Le th\'eor\`eme~\ref{rig} montre que c'est bien le cas et que, de plus, la période $h$ des structures de Frobenius fortes peut \^etre major\'ee explicitement et ind\'ependamment du nombre premier $p$. Comme nous l'avons d\'ej\`a mentionn\'e,  ce dernier point est essentiel au vu de la conjecture~\ref{1} et du th\'eor\`eme~\ref{algebrique}. Comme nous l'a indiqué Gilles Christol, il semble que l'on puisse obtenir l'existence d'une structure de Frobenius forte pour n'importe quel sous module $N$ d'un module différentiel $M$ associé à une équation 
 de Picard-Fuchs. En effet, d'après le théorème 4.2.6 \cite{delignehodge}, $M$ est semi-simple et 
 donc $N$ est semi-simple et on peut le d\'ecomposer comme somme directe finie de sous-modules simples $N_i$. 
 Rappelons que d'apr\`es le Theorem 22.2.1 de \cite{Kedlaya}, $M$ a une structure de Frobenius forte pour presque tout $p$. Soit $p$ l'un de ces nombres premiers et $h$ la p\'eriode de la structure de Frobenius forte correspondante. 
 En notant $\phi$ le Frobenius, on a donc que $\phi^{h}(M)$ et $M$ sont isomorphes. 
 Donc, pour tout $m\geq1$, $\phi^{mh}(N_i)$ est sous module simple de $\phi^{h}(M)$ et on obtient l'existence 
 de deux entiers  $m$ et $m'$ tels que $\phi^{mh}(N_i)$=$\phi^{m'h}(N_i)$. 
Autrement dit, $N_i$ est muni d'une structure de Frobenius forte pour $p$.  On en d\'eduit que $N$ est muni d'une structure de Frobenius forte pour $p$ car l'application $\phi$ respecte les sommes directes. 
Malheureusement, cet argument ne permet pas de majorer pr\'ecis\'ement la p\'eriode $h$ de cette structure de Frobenius forte, ni m\^eme de la majorer ind\'ependamment de $p$.  L'int\'er\^et principal du théorème~\ref{rig} est 
qu'il permet justement de majorer la p\'eriode des structures de Frobenius fortes des opérateurs rigides 
ind\'ependamment de $p$. 
Le th\'eor\`eme~\ref{rig}  a \'egalement \'et\'e obtenu sous une forme l\'eg\`erement diff\'erente, mais essentiellement \'equivalente, par Crew 
\cite{crew}.  Il nous semble n\'eanmoins utile de 
pr\'esenter la d\'emonstration propos\'ee ici qui est \'ecrite dans un langage diff\'erent et plus \'el\'ementaire,  
notamment exempt de toute consid\'eration cohomologique. 
Notons qu'un cas particulier du théorème 1.5 de \cite{esnault} implique aussi qu'un op\'erateur diff\'erentiel satisfaisant aux conditions du théorème~\ref{rig} a une structure de Frobenius forte pour presque tout $p$.  
Mais ce dernier r\'esultat ne donne pas explicitement de renseignement sur la période des structures de Frobenius associ\'ees et ne permet donc pas d'utiliser le th\'eor\`eme~\ref{algebrique} comme nous le faisons.  

Avant d'\'enoncer   le théorème~\ref{rig}, considérons $L$ défini comme en \eqref{L} et supposons que les exposants aux points singuliers r\'eguliers sont des nombres rationnels. Soit $s$ la valuation de $a_0(z)$. Considérons les ensembles suivants:  $\mathfrak{A}_1$  est form\'e du terme constant du polynôme $\frac{a_0(z)}{z^s}$, le coefficient leader de $a_0(z)$ et du discriminant de $a_0(z)$, $\mathfrak{A}_2$ 
est form\'e des dénominateurs des exposants aux points singuliers réguliers de $L$ et  $\mathfrak{A}_3=\left\{\frac{a_i(z)}{a_0(z)}\right\}_{1\leq i\leq n}$. Soit $d$ le plus petit commun multiple 
des d\'enominateurs des exposants de $L$ en les points singuliers réguliers. 
On pose $h_1=\phi(d)$, où $\phi$ est la fonction indicatrice de Euler, $h_2$ la dimension du corps de décomposition du polynôme $a_0(z)$ sur $\mathbb{Q}$, et finalement $h=h_1h_2$.

\begin{theo}\label{rig}
	Soit $L\in\mathbb{Q}[z][d/dz]$ un op\'erateur diff\'erentiel d\'efini comme en \eqref{L}.
	Supposons que les conditions suivantes sont v\'erifi\'ees. 
	\begin{enumerate}
		\item Les points singuliers de $L$ sont r\'eguliers, c'est-\`a dire que $L$ est fuchsien.
		
		\item Les exposants aux points singuliers r\'eguliers sont des nombres rationnels.
		
		\item Le groupe de monodromie de $L$ est rigide.
	\end{enumerate}
	Soit $\mathcal{S}$ l'ensemble des nombres premiers tels que $a_0(z)\in\mathbb{Z}_{(p)}[z]$, tout élément de $\mathfrak{A}_1$ et $\mathfrak{A}_2$ ait une norme $p$-adique égale à 1 et tout élément de $\mathfrak{A}_3$ ait une norme de Gauss inférieure ou égale à 1.  Alors pour tout $p\in\mathcal{S}$, l'opérateur différentiel a une structure de Frobenius forte de période $h$.
\end{theo}

Dans \cite{crew},  Crew montre que, sous les hypoth\`eses 1, 2 et 3 du th\'eor\`eme, si $p$ est un nombre premier tel que $L$ définit un isocristal \emph{surconvergent} et v\'erifiant certaines autres hypoth\`eses (les conditions $C_1$ et $C_3$ dans \cite{crew}), alors $L$ a une structure de Frobenius forte pour $p$. 
Sa preuve repose sur des outils de cohomologie $p$-adique et le fait que la surconvergence lui permet (Theorem 1 et Theorem 2 de \cite{crew}) de définir  la rigidit\'e $p$-adique en termes de la cohomologie $p$-adique. 
Il montre aussi que la p\'eriode $h$ obtenue ne d\'epend pas de $p$. 
L'int\'er\^et de notre approche est son aspect plus \'el\'ementaire puisqu'elle repose sur les aspects classiques de   la théorie des équation différentielles (\`a la fois sur 
$\mathbb{C}(z)$  et $p$-adique) et que nous donnons une description précise de l'ensemble des nombres premiers $p$ qui munissent $L$ d'une structure de Frobenius forte. 

\begin{remarque}
Comme nous l'avons  déjà mentionné, si $L$ est muni d'une structure de Frobenius forte pour $p$, alors son rayon de convergence au point générique est égal à 1.  Cela implique que $L$ définit un isocristal surconvergent. Donc, sous les hypothèses du théorème~\ref{rig}, on obtient a posteriori que si $p\in\mathcal{S}$, alors $L$ définit bien un isocristal surconvergent. 
\end{remarque}

\section{Démonstration du théorème \ref{algebrique}.}\label{algebricite}

Cette partie est consacr\'ee \`a l'algebricit\'e modulo $p$ des solutions des \'equations diff\'erentielles 
poss\'edant une structure de Frobenius forte pour le nombre premier $p$. Nous d\'emontrons 
le th\'eor\`eme~\ref{algebrique}.  

Rappelons tout d'abord que l'ensemble $\mathcal{W}_{p}$ est compos\'e  des s\'eries 
de la forme 
\[f(z)=\sum_{n\in\mathbb{Z}}a_{n}z^{n},\]
dont les coefficients appartiennent \`a $\mathbb C_p$ et telles que  la famille $\{|a_{n}|\}_{n\in\mathbb{Z}}$ est born\'ee et tend vers 0 lorsque $n$ tend n\'egativement vers l'infini. 
L'anneau $\mathcal{W}_{p}$ est complet pour la norme \[|f|=\sup_{n\in\mathbb{Z}}|a_{n}|.\] 
Par construction, $E_{p}$ est un sous-corps de $\mathcal{W}_{p}$ et on note $\vartheta_{E_{p}}$ les \'el\'ements de $E_{p}$ dont la norme est inf\'erieure ou \'egale \`a 1.

Nous commen\c{c}ons par montrer les deux lemmes suivants.

\begin{lem}\label{reduction}
	Soit $\mathfrak{m}$ l'id\'eal maximal de $\vartheta_{E_{p}}$. Il existe un isomorphisme de corps \[\phi:\vartheta_{E_{p}}/\mathfrak{m}\rightarrow\overline{\mathbb{F}}_{p}(z).\]
\end{lem}

\begin{proof}
	Soient $\vartheta_{\mathbb{C}_{p}}$ l'anneau des entiers de $\mathbb{C}_{p}$ et  $\mathfrak{M}$ son id\'eal maximal. Soit $x\in\vartheta_{\mathbb{C}_{p}}$, comme $x$ est la limite d'\'el\'ements dans la cl\^oture alg\'ebrique de $\mathbb{Q}_{p}$, il existe une extension finie $K$ de $\mathbb{Q}_{p}$ et $y\in K$ tels que $|x-y|<1$. Ainsi, $|y|\leq1$ et dans $\vartheta_{\mathbb{C}_{p}}/\mathfrak{M}$, $\overline{x}=\overline{y}$. Comme le corps r\'esiduel de $K$ est une extension finie de $\mathbb{F}_{p}$, on a $\overline{y}\in\overline{\mathbb{F}}_p$.
	On d\'efinit \[\omega:\vartheta_{\mathbb{C}_{p}}\rightarrow\overline{\mathbb{F}}_{p}\]
	par $\omega(x)=\overline{y}$. Notons que $\omega(x)$ ne d\'epend pas de $y$ et est un homomorphisme d'annaux. D'apr\`es le lemme de Hensel, $\omega$ est surjectif et si $x\in\mathfrak{M}$, alors $\omega(x)=0$. Ainsi,
	\[\omega:\vartheta_{\mathbb{C}_{p}}/\mathfrak{M}\rightarrow\overline{\mathbb{F}}_{p}\]
	est un isomorphisme tel que $\omega(\overline{x})=\overline{x}$ pour tout $x\in\mathbb{Z}_{p}$. Soit $\vartheta_{\mathcal{W}_{p}}$ l'ensemble des \'el\'ements de $\mathcal{W}_{p}$ dont la norme est inf\'erieure ou \'egale \`a 1. Comme la norme est ultram\'etrique, 
	$\vartheta_{\mathcal{W}_{p}}$ est un anneau et, si $J$ d\'esigne l'ensemble des \'el\'ements de $\mathcal{W}_{p}$ dont la norme est strictement inf\'erieure \`a $1$, alors $J$ est un id\'eal de $\vartheta_{\mathcal{W}_{p}}$ et le quotient de $\vartheta_{\mathcal{W}_{p}}$ par $J$ est isomorphe \`a $\overline{\mathbb{F}}_{p}((z))$. En effet, soit $f(z)=\sum_{n\in\mathbb{Z}}a_{n}z^{n}\in\vartheta_{\mathcal{W}_{p}}$. Alors, pour tout entier $n$, $|a_{n}|\leq1$ et par d\'efinition de l'anneau $\mathcal{W}_{p}$ il existe un nombre naturel $N$ tel que, pour tout $n<-N$, $|a_{n}|<1$. Ainsi, on a \[\overline{f(z)}:=\sum_{n\geq-N}\overline{a_{n}}z^{n}\in(\vartheta_{\mathbb{C}_{p}}/{\mathfrak{M}})((z)).\] 
	Consid\'erons l'application $\widetilde{\omega}:\vartheta_{\mathcal{W}_{p}}/J\rightarrow\overline{\mathbb{F}}_{p}((z))$ 
	d\'efinie par 
	$\widetilde{\omega}(\overline{f})=\sum_{n\geq-N}\omega(\overline{a_{n}})z^{n}$. Si $f\in\mathbb{Z}_{p}[[z]]$, 
	alors $\widetilde{\omega}(\overline{f})=\overline{f}$. Comme $\omega$ est un isomorphisme il en est de m\^eme pour $\widetilde{\omega}$. Construisons \`a pr\'esent l'isomorphisme $\phi$. 
	Comme $\vartheta_{E_{p}}$ est un anneau local et un sous-anneau de $\vartheta_{\mathcal{W}_{p}}$, si $\mathfrak{m}$ d\'esigne  l'id\'eal maximal de $\vartheta_{E_{p}}$, alors $\mathfrak{m}\subset J$.  
	Ainsi,  $\widetilde{\omega}$ peut se restreindre  \`a $\vartheta_{E_{p}}/\mathfrak{m}$. 
	
	Par construction, si $h\in\vartheta_{E_{p}}$, 
	$\overline{h}\in(\vartheta_{\mathbb{C}_{p}}/{\mathfrak{M}})(z)$, d'o\`u $\overline{h}s=t$, o\`u $s,t\in(\vartheta_{\mathbb{C}_{p}}/{\mathfrak{M}})[z]$ sont diff\'erents du polyn\^ome nul.
	Nous pouvons donc d\'efinir, $\phi(\overline{h})=\widetilde{\omega}(t)/\widetilde{\omega}(s)\in\overline{\mathbb{F}}_p(z)$.
\end{proof}

Le lemme suivant est  d\'emonstr\'e dans \cite[Proposition 6.2]{Borisgfonct} dans le cas o\`u $L=\mathbb{C}$. En utilisant le m\^eme argument, on obtient le lemme suivant.

\begin{lem}\label{descend}
	Soit $K$ un corps,  $L$ une extension de $K$ et  $f_{1},\ldots, f_{n}\in K[[z]]$. 
	S'il existe des polyn\^omes $a_{1}(z),\cdots, a_{n}(z)\in L[z]$, non tous nuls, tels que \[a_{1}(z)f_{1}+\cdots+a_{n}(z)f_{n}=0,\]
	alors il existe des polyn\^omes $c_{1}(z),\cdots, c_{n}(z)\in K[z]$, non tous nuls,  tels que \[c_{1}(z)f_{1}+\cdots+c_{n}(z)f_{n}=0.\] 
\end{lem}

Nous pouvons maintenant d\'emontrer le th\'eor\`eme~\ref{algebrique}. 

\begin{proof} [D\'emonstration du th\'eor\`eme~\ref{algebrique}]
	Fixons un nombre premier $p$. Supposons que $A$ est la matrice compagnon de l'op\'erateur diff\'erentiel
	\begin{equation}\label{001}
	L:=\frac{d^{n}}{dz^{n}}+a_{1}(z)\frac{d^{n-1}}{dz^{n-1}}+\cdots+a_{n-1}(z)\frac{d}{dz}+a_{n}(z),
	\end{equation}
	o\`u les $a_{i}(z)$ sont des fractions rationnelles \`a coefficients dans $\mathbb{Q}$. 
	Comme $A$ a une structure de Frobenius forte de p\'eriode $h$, il existe $H\in {\rm GL}_n(E_{p})$ telle que $$\frac{d}{dz}H=AH-H(p^{h}z^{p^{h}-1}A(z^{p^{h}})).$$ Soit  $g\in\mathcal{W}_{p}$ tel que $Lg=0$, alors le vecteur $\vec{y}=(g,g',\ldots,g^{(n-1)})^{T}$ est solution du syst\`eme
	\begin{equation}\label{sys}
	\frac{d}{dz}Y=AY.
	\end{equation}
	Ainsi, le vecteur $\vec{y}(z^{p^{h}})$ est solution du syst\`eme \[\frac{d}{dz}Y=p^{h}z^{p^{h-1}}A(z^{p^{h}})Y.\]
	Par cons\'equent, le vecteur $H\vec{y}(z^{p^{h}})$ est solution du syst\`eme~\eqref{sys}. Comme $E_{p}\subset\mathcal{W}_{p}$, on trouve que $H\vec{y}(z^{p^{h}})$ est un vecteur \`a coefficients dans $\mathcal{W}_{p}$, dont le premier coefficient est une solution de l'\'equation associ\'eé \`a l'op\'erateur $\eqref{001}$. 
	Soit $V:=\{(g,g',\ldots,g^{n-1})\,:\,g\in\mathcal{W}_{p},Lg=0\}$. Il s'agit d'un $\mathbb{C}_{p}$-espace vectoriel. 
	Nous avons  donc construit un endomorphisme
	\begin{align*}
	\psi:V\rightarrow& V \\
	\vec{y}\mapsto& H\vec{y}(z^{p^{h}}).
	\end{align*}
	Comme $\mbox{dim}_{\mathbb{C}_{p}}V=r\leq n$, le th\'eor\`eme de Cayley-Hamilton assure l'existence de 
	$c_{0},\ldots,c_{r-1}\in\mathbb{C}_{p}$ tels que
	\begin{equation}\label{cayley}
	\psi^{r}+c_{r-1}\psi^{r-1}+\cdots+c_{1}\psi+c_{0}=0.
	\end{equation}
	Comme par hypothèse $f(z)\in\mathbb{Z}_{(p)}[[z]]$ et $L$ est annul\'ee par $f(z)$, alors le vecteur $\vec{w}=(f,f',\ldots,f^{(n-1)})$ est dans $V$. Consid\'erons \`a pr\'esent $Z$, le $E_{p}$-espace vectoriel engendr\'e par les \'el\'ements de l'ensemble $\{f^{(j)}(z^{p^{ih}})\,:\, j\in\{0,\ldots,n-1\}, i\in\mathbb{N}\}$. L'\'egalit\'e $\eqref{cayley}$ montre que $Z$ a pour dimension au plus $nr$. 
	Comme $f(z),\ldots, f(z^{p^{nrh}})\in Z$, il existe $j\leq nr$ et $b_{0},\ldots, b_{j}\in E_p$ 
	tels que  
	\[b_{j}f(z^{p^{jh}})+b_{j-1}f(z^{p^{(j-1)h}})+\cdots+b_{0}f(z)=0.\]
	Soit $b_{l}$ tel que $|b_{l}|=\max\{|b_{0}(z)|,\ldots,|b_{j}(z)|\}$ et soit $c_{i}(z)=b_{i}(z)/b_{l}(z)$. Alors, pour tout $i\in\{0,\ldots,j\}$, $|c_{i}|\leq1$ et 
	\[c_{j}f(z^{p^{jh}})+c_{j-1}f(z^{p^{(j-1)h}})+\cdots+c_{0}f(z)=0.\]
	On pose $d_{i}(z)=\overline{c_{i}(z)}$, alors on a
	\begin{equation}\label{10}
	d_{j}(z)(f_{\mid p}(z^{p^{jh}}))+d_{j-1}(z)(f_{\mid p}(z)^{p^{(j-1)h}})+\cdots+d_{0}(z)f_{\mid p}(z)=0.
	\end{equation}
	et $d_0(z),\ldots, d_j(z)$ ne sont toutes nulles car $1=\max{|c_0(z)|,\ldots,|c_j(z)|}$. Soient $\widetilde{\omega}$ l'homomorphisme construit dans le lemme~\ref{reduction} et 
	$t_{i}(z)=\widetilde{\omega}(d_{i}(z))$. Puisque $\widetilde{\omega}$ est un isomorphisme alors $t_0(z),\ldots, t_j(z)$ ne sont pas toutes nulles. Comme $\widetilde{\omega}(f_{\mid p})=f_{\mid p}$ car $f\in\mathbb{Z}_{(p)}[[z]]$, alors $\eqref{10}$ 
	implique que 
	\[t_{j}(z)(f_{\mid p}(z^{p^{jh}}))+t_{j-1}(z)(f_{\mid p}(z^{p^{(j-1)h}}))+\cdots+t_{0}(z)f_{\mid p}(z)=0.\] 
	Finalement, le 
	lemme~\ref{descend} assure l'existence des polyn\^omes $r_{0}(z),\ldots,r_{j}(z)\in\mathbb{F}_{p}(z)$, non tous nuls, 
	tels que 
	\[r_{j}(z)(f_{\mid p}(z)^{p^{jh}})+r_{j-1}(z)(f_{\mid p}(z^{p^{(j-1)h}}))+\cdots+r_{0}(z)f_{\mid p}(z)=0.\]
	Comme les coefficients de $f_{\mid p}(z)$ sont dans $\mathbb F_p$, 
	la derni\`ere expression devient \[r_{j}(z)(f_{\mid p}(z))^{p^{jh}}+r_{j-1}(z)(f_{\mid p}(z))^{p^{(j-1)h}}+\cdots+r_{0}(z)f_{\mid p}(z)=0.\]
	Et puisque $j\leq nr\leq n^{2}$, nous obtenons que le degr\'e de $f_{\mid p}$ sur $\mathbb{F}_{p}(z)$ 
	est major\'e par $p^{n^{2}h}$, comme souhait\'e. 
\end{proof}

\begin{remarque}\label{r}
	La d\'emonstration montre plus pr\'ecis\'ement que le degr\'e d'alg\'ebricit\'e de $f_{\mid p}(z)$ est born\'e par $p^{nrh}$, 
	o\`u $r$ d\'esigne la dimension du $\mathbb{C}_{p}$-espace vectoriel $V$. 
\end{remarque}


\section{D\'emonstration du th\'eor\`eme \ref{rig}}\label{sFfrig}

Nous rappelons que $Frob:\mathbb{C}_{p}\rightarrow\mathbb{C}_{p}$ est l'automorphisme de Frobenius de $\mathbb{C}_{p}$ choisi dans la remarque~\ref{frobcp}. \`A pr\'esent, nous pr\'esentons les diff\'erentes \'etapes de la d\'emonstration du th\'eor\`eme~\ref{rig}. Celles-ci seront d\'emontr\'ees dans la partie~\ref{dempas}.

\smallskip

\textbf{Premier pas}. Le premier pas est consacr\'e au calcul des groupes de monodromie locale de $L$ à conjugaison près. Plus pr\'ecis\'ement, soit $A$ la matrice compagnon de $L$ et soient $\gamma_{1}=0,\ldots, \gamma_{r}=\infty$ les points singuliers r\'eguliers de $L$. Pour tout $j\in\{1,\ldots,r\}$, soit $M(A,\gamma_{j})$ la matrice de monodromie locale de $A$ en $\gamma_{j}$. À l'aide du lemme~\ref{monodromie} nous allons montrer qu'elle est conjugu\'ee \`a une matrice de la forme $\exp(2\pi iC_{j})$, o\`u  $C_{j}\in {\mathcal M}_n(\mathbb{C})$ satisfait aux conditions suivantes:

\textbf{(a)} Si $\lambda,\beta$ sont deux valeurs propres diff\'erentes de $C_{j}$, alors $\lambda-\beta\notin\mathbb{Z}$.

\textbf{(b)} L'ensemble des exposants de $L$ en $\gamma_{j}$ et l'ensemble des valeurs propres de $C_j$ sont égaux modulo $\mathbb{Z}$.


\smallskip
 
\textbf{Deuxi\`eme pas}. Nous montrerons que pour $p\in\mathcal{S}$ et tout $j\in\{1,\ldots,r\}$, les matrices $\exp(2\pi iC_{j})$ et $\exp(2\pi ip^{h}C_{j})$ sont conjugu\'ees. Nous rappelons que l'ensemble de nombres premiers 
$\mathcal{S}$ et l'entier $h$ sont définis comme dans le théorème~\ref{rig}.

\smallskip

\textbf{Troisi\`eme pas}. Nous montrerons que pour tout $p\in\mathcal{S}$, la matrice $B=p^{h}z^{p^{h}-1}A(z^{p^{h}})$ est $E_{p}$-\'equivalente \`a la matrice $A$.  La matrice $B$ est dans ${\mathcal M}_n(E_{p})$. La d\'emonstration de ce troisi\`eme pas se d\'ecompose selon les 3 \'etapes suivantes:

\smallskip

\textbf{(i)} Soit $m$ entier strictement positif tel que $-m$ n'est pas une singularité de $L$.  Nous montrerons que $B$ est $E_{p}$-\'equivalente \`a $$G=\frac{\gamma_{1}+m}{(z+m)(z-\gamma_{1})}F_{1}+\cdots+\frac{\gamma_{r-1}+m}{(z+m)(z-\gamma_{r-1})}F_{r-1}-\frac{1}{z+m}F_{r},$$ o\`u $F_{j}\in {\mathcal M}_n(\mathbb{C}_{p})$ pour $j\in\{1,\ldots,r\}.$

\smallskip
Dans la démonstration de \textbf{(ii)} nous utiliserons un isomorphisme de corps $\kappa:\mathbb{C}_{p}\rightarrow\mathbb{C}$. L'isomorphisme $\kappa$ s'\'etend naturellement en un isomorphisme de corps entre $\mathbb{C}_{p}(z)$ et $\mathbb{C}(z)$, que nous notons encore $\kappa:\mathbb{C}_{p}(z)\rightarrow\mathbb{C}(z)$. La matrice $M^{\kappa}\in {\mathcal M}_n(\mathbb{C}(z))$ d\'esigne la matrice que l'on obtient en appliquant $\kappa$ \`a chaque entr\'ee de $M\in {\mathcal M}_{n}(\mathbb{C}_p(z))$. 

\smallskip

\textbf{(ii)} Nous montrerons d'une part que la matrice de monodromie locale de $G^{\kappa}$ en $\kappa(\gamma_{j})$, $M(G^{\kappa},\kappa(\gamma_{j}))$, est conjugu\'ee \`a la matrice $\exp(2\pi i p^{h}C_{j})$. D'autre part, nous prouvons que la matrice de monodromie locale en $-m$ est l'identit\'e.

D'apr\`es le deuxi\`eme pas, les matrices $\exp(2\pi i C_{j})$ et $\exp(2\pi ip^{h}C_{j})$ sont conjugu\'ees. Ainsi, il d\'ecoule du premier pas que pour $j\in\{1,\ldots, r\}$, les matrices $M(A,\gamma_{j})$ et $M(G^{\kappa},\kappa(\gamma_{j}))$ sont conjugu\'ees.

\smallskip

\textbf{(iii)} Comme la monodromie de $\frac{d}{dz}y=Ay$ est rigide, les conjugaisons des matrices $M(A,\gamma_{j})$ et $M(G^{\kappa},\kappa(\gamma_{j}))$ entraînent qu'il existe $U\in {\rm GL}_n(\mathbb{C})$ telle que $M(A,\gamma_{j})=UM(G^{\kappa},\kappa(\gamma_{j}))U^{-1}$ pour $1\leq j\leq r$. Il suit donc que les groupes de monodromie des systèmes $Ay=\frac{d}{dz}y$ et $G^{\kappa}y=\frac{d}{dz}y$ sont isomorphes. D'apr\`es la correspondance de Riemman Hilbert, théorème 6.15 de \cite{Singer}, on obtient qu'il existe $H_{1}\in {\rm GL}_n(\mathbb{C}(z))$ telle que \[\frac{d}{dz}H_{1}=AH_{1}-H_{1}G^{\kappa}.\]
On pose $H=H_{1}^{\kappa^{-1}}$, ainsi $H\in {\rm GL}_n(\mathbb{C}_{p}(z))\subset {\rm GL}_n(E_{p})$ et \[\frac{d}{dz}H=A^{\kappa^{-1}}H-HG.\]
Remarquons que $A^{\kappa^{-1}}=A$ car $A$ est une matrice \`a coefficients dans $\mathbb{Q}(z)$. Par cons\'equent, $A$ et $G$ sont $E_{p}$-\'equivalentes. Or, d'apr\`es \textbf{(i)}, on sait que $G$ est $E_{p}$-\'equivalente \`a $B$. Par transitivit\'e on conclut que $A$ et $B$ sont $E_{p}$-\'equivalentes.

\subsection{Disque non singulier et singulier r\'egulier}\label{sinpadic}

Dans cette partie nous \'enon\c cons et démontrons le théorème~\ref{dmfs} ci-dessous. 
Ce résultat nous permettra dans la d\'emonstration du th\'eor\`eme~\ref{rig} de passer de la th\'eorie des \'equations diff\'erentielles p-adiques \`a la th\'eorie des \'equations diff\'erentielles classiques.

Pour tout $\gamma\in\mathbb{C}_p$, nous d\'esignons par $D_{\gamma}$  le disque ouvert de centre $\gamma$ et de rayon 1 et par $D_{\infty}$ l'ensemble  
des \'el\'ements de $\mathbb{C}_p$ dont la norme est strictement sup\'erieure \`a $1$ et $\infty\in D_{\infty}$.
Nous notons $\vartheta_{\mathbb{C}_p}$ l'ensemble des  \'el\'ements de $\mathbb{C}_p$ dont la norme est inf\'erieure ou \'egale \`a 1 et si $\alpha,\gamma\in\vartheta_{\mathbb{C}_p}$ alors $D_{\alpha}=D_{\gamma}$ si, et seulement si $|\alpha-\gamma|<1$. Et, $D_{\alpha}\cap D_{\gamma}=\emptyset$ si, et seulement si $|\alpha-\gamma|=1$.
Nous notons $E(D_{\gamma})$ le compl\'et\'e pour la norme de Gauss des fractions rationnelles qui n'ont pas de p\^ole dans la boule 
$D_{\gamma}$. 

Pour \'enoncer le th\'eor\`eme~\ref{dmfs} nous aurons besoin de rappeler quelques notions de la th\'eorie des \'equations diff\'erentielles p-adiques. Pour cela nous reprenons l'exposition de \cite{Gillesfacteurs}. 

Soient $A\in {\mathcal M}_n(E_p)$ et $\gamma\in\vartheta_{\mathbb{C}_p}$. 
Nous dirons que la matrice $A$ est \emph{non singuli\`ere dans le disque}  $D_{\gamma}$, respectivement \`a l'infini, s'il existe une matrice $A_{\gamma}$, respectivement $A_{\infty}$, $E_p$-\'equivalente \`a $A$ et qui appartient \`a ${\mathcal M}_n(E(D_{\gamma}))$, respectivement \`a $z^{2}{\mathcal M}_n(E(D_{\infty}))$. 
Nous dirons que la matrice $A$ est \emph{singuli\`ere r\'eguli\`ere dans le disque} $D_{\gamma}$, respectivement \`a l'infini, s'il existe un point $\beta$ de $D_{\gamma}$ et une matrice $A_{\gamma}$, respectivement $A_{\infty}$, $E_p$-\'equivalente \`a $A$ tels que $(z-\beta)A_{\gamma}$ appartient \`a ${\mathcal M}_n(E(D_{\gamma}))$, respectivement $zA_{\infty}$ appartient \`a ${\mathcal M}_n(E(D_{\infty}))$.

\begin{theo}\label{dmfs}
	Soient $B\in {\mathcal M}_n(E_p)$ et $S=\{\gamma_1,\ldots,\gamma_{r-1}\}\subset\vartheta_{\mathbb{C}_{p}}$ tel que pour tout $i\neq j$, $|\gamma_{i}-\gamma_{j}|_p=1$. Supposons que  pour tout $\gamma\in S\cup\{\infty\}$, la matrice $B$ est singulière régulière dans le disque $D_\gamma$ et que pour tout $\gamma\notin{S}\cup\{\infty\}$ la matrice $B$ est non singulière dans le disque $D_{\gamma}$. Soit $m$ un entier strictement positif tel que $-m\notin S$. Pour tout $\gamma\in S$, soient $\beta_{\gamma}$ un point de $D_{\gamma}$ et $B_{\gamma}$ une matrice $E_p$-\'equivalente \`a $B$ tels que $(z-\beta_{\gamma})B_{\gamma}$ appartient \`a ${\mathcal M}_n(E(D_{\gamma}))$. Soit $B_{\infty}$ une matrice $E_p$-\'equivalente \`a $B$, telle que $zB_{\infty}$ appartient \`a ${\mathcal M}_n(E(D_{\infty}))$.
	Alors $B$ est $E_p$-\'equivalente \`a une matrice $G$ de la forme 
	\begin{equation}\label{fuchs}
	G=\frac{\beta_{\gamma_{1}}+m}{(z+m)(z-\beta_{\gamma_{1}})}G_{1}+\cdots+\frac{\beta_{\gamma_{r-1}}+m}{(z+m)(z-\beta_{\gamma_{r-1}})}G_{r-1}-\frac{1}{z+m}G_{r},
	\end{equation}
	o\`u,  pour $j\in\{1,\ldots,r-1\}$, $G_{j}$ est une matrice à coefficients dans $\mathbb{C}_p$  semblable à  $[(z-\beta_{\gamma})B_{\gamma}](\beta_{\gamma})$, $G_r$ est une matrice à coefficients dans $\mathbb{C}_p$ semblable à $[-zB_{\infty}](\infty)$ et $\sum_{\gamma\in S\cup\{\infty\}}G_{\gamma}$ est une matrice diagonale à coefficients dans $\mathbb{Z}$.
\end{theo}
Le théorème~\ref{dmfs} est démontré dans \cite[théorème 4.2]{Gillesfacteurs} dans le cas où le disque $D_{\infty}$ est non singulier.

\begin{proof}
	Considérons la matrice $$V=-\frac{1}{z^{2}}B\left(\frac{1-zm}{z}\right),$$
	obtenue en appliquant le changement de variables $z\mapsto \frac{1-zm}{z}$ au syst\`eme $By=\frac{d}{dz}y$. Le disque $D_{\infty}$ est non singulier pour la matrice $V$ car le disque $D_{-m}$ est non singulier pour la matrice $B$. Ainsi nous ramenons toutes les singularit\'es \`a distance finie. 
	Posons pour $j\in\{1,\ldots,r-1\}$, $\tau_{j}=1/(\beta_{\gamma_{j}}+m)$ et $\tau_{r}=0$. Notons que pour $j\in\{1,\ldots, r-1\}$,  $|\beta_{\gamma_j}+m|_p=1$ car $D_{-m}$ est un disque non singulier pour la matrice $B$.
	Consid\'erons l'ensemble $S'=\{\tau_{1},\ldots,\tau_{r-1},0\}$. Nous allons appliquer le théorème 4.2 de \cite{Gillesfacteurs} \`a la matrice $V$ et \`a l'ensemble $S'$. Nous allons donc montrer que $V$ et $S'$ satisfont aux hypoth\`eses du théorème 4.2 de \cite{Gillesfacteurs}. D'apr\`es le changement de variables choisi, il nous reste \`a v\'erifier que:
	\begin{enumerate}
		\item $|\tau_{i}-\tau_j|=1$ pour $i\neq j$;
		\item Soit $\eta$ tel que $|\eta|\leq1$ et $|\tau_i-\eta|=1$ pour $i\in\{1,\ldots,r\}$. Alors $V$ est non singuli\`ere dans le disque $D_{\eta}$;
		\item Les disques singuliers r\'eguliers de $V$ sont exactement les disques $D_{\tau_{j}}$, $1\leq j\leq r-1$, et $D_{0}$.
	\end{enumerate}
		\emph{Preuve du point 1.} On a
			$$\tau_{i}-\tau_j=\frac{1}{\beta_{\gamma_{i}}+m}-\frac{1}{\beta_{\gamma_j}+m}=\frac{\beta_{\gamma_{j}}-\beta_{\gamma_{i}}}{(\beta_{\gamma_{i}}+m)(\beta{\gamma_j}+m)}.$$  Par hypothèse $|\gamma_{j}-\gamma_{i}|=1$ alors $|\beta_{\gamma_{i}}-\beta_{\gamma_{j}}|=1$ et on a déjà vu que $|(\beta_{\gamma_{j}}+m)|=1$ pour $j\in\{1,\ldots,r-1\}$, d'où $|(\beta_{\gamma_{i}}+m)(\gamma_j+m)|=1$. Ainsi $|\tau_{i}-\tau_j|=1$.
			
			\smallskip
			
			\emph{Preuve du point 2.} Comme $|\tau_{i}-\eta|=1$ pour tout $i\in\{1,\ldots, r\}$, alors on a $|\eta|=1$ car $\tau_{r}=0$.
			Supposons que $V$ est singuli\`ere dans le disque $D_{\eta}$, alors la matrice $B$ est singulière dans le disque $D_{\frac{1-\eta m}{\eta}}$. Donc, il existe $i\in\{1,\ldots, r-1\}$ tel que 
		 \[\overline{\beta_{\gamma_{i}}}=\overline{\left(\frac{1-\eta m}{\eta}\right)}\]
			et comme $\beta{\gamma_i}=\frac{1-\tau_im}{\tau_i}$ il suit que \[\overline{\left(\frac{1-\eta m}{\eta}\right)}=\overline{\left(\frac{1-\tau_im}{\tau_i}\right)}.\]
			Ainsi, on a $\overline{\eta}=\overline{\tau_{i}}$. Autrement dit $|\tau_{i}-\eta|<1$, ce qui est une contradiction.
			
			\smallskip
			
			\emph{Preuve du point 3.} Par hypothèse pour chaque $j\in\{1,\ldots,r-1\}$, la matrice $B$ est $E_{p}$-\'equivalentes \`a $B_{\gamma_j}$, où $(z-\beta_{\gamma_j})B_{\gamma_j}$ appartient à ${\mathcal M}_{n}(E(D_{\gamma_j}))$. Ainsi, $B_{\gamma_j}=\frac{1}{z-\beta_{\gamma_j}}H_{\gamma_j}$, où $H_{\gamma_j}$ appartient à ${\mathcal M}_{n}(E(D_{\gamma_j}))$. Montrons que $V$ est $E_{p}$-\'equivalente \`a $$V_{j}=-\frac{1}{z}\cdot\frac{1}{1-z(m+\beta_{\gamma_{j}})}H_{\gamma_j}\left(\frac{1-zm}{z}\right).$$
			En effet, on sait qu'il existe $G_{j}\in {\rm GL}_n(E_{p})$ telle que \[\frac{d}{dz}G_j=BG_j-G_jB_{\gamma_j}.\]
			On pose $T_j:=G_j\left(\frac{1-zm}{z}\right)\in {\rm GL}_n(E_p)$, donc en appliquant la d\'eriv\'ee d'une composition on obtient que
			\begin{equation}
			\begin{split}
			\frac{d}{dz}T_j&=\left[B\left(\frac{1-zm}{z}\right) G_j\left(\frac{1-zm}{z}\right)-G_j\left(\frac{1-zm}{z}\right)B_{\gamma_j}\left(\frac{1-zm}{z}\right)\right]\left(\frac{-1}{z^2}\right)\\
			&=VT_j-T_j\cdot\left(\left(\frac{-1}{z^2}\right)B_{\gamma_j}\left(\frac{1-zm}{z}\right)\right).
			\end{split}
			\end{equation}
			Mais $B_{\gamma_j}=\frac{1}{z-\beta_{\gamma_j}}H_{\gamma_j}$ entraîne que $B_{\gamma_j}\left(\frac{1-zm}{z}\right)=\frac{z}{1-z(m+\beta_{\gamma_j})}H_{\gamma_{j}}\left(\frac{1-zm}{z}\right)$, d'où $\left(\frac{-1}{z^2}\right)B_{\gamma_j}\left(\frac{1-zm}{z}\right)=V_j$ et ainsi $$\frac{d}{dz}T_j=VT_j-T_jV_j.$$
			Donc $V$ et $V_j$ sont $E_p$-\'equivalentes.
			
			Comme $m+\beta_{\gamma_{j}}=1/\tau_{j}$ pour $1\leq i\leq j-1$, il vient que
			\begin{equation}\label{vj}
			V_{j
			}=-\frac{1}{z}\cdot\frac{\tau_j}{\tau_{j}-z}H_{\gamma_{j}}\left(\frac{1-zm}{z}\right).
			\end{equation}
			
			Enfin, comme $B$ est $E_{p}$-\'equivalente \`a $B_{\infty}$, où $B_{\infty}=\frac{1}{z}H_{\infty}$ avec $H_{\infty}$ dans ${\mathcal M}_{n}(E(D_{\infty}))$ alors $V$ est $E_{p}$-\'equivalente \`a
			\begin{equation}\label{vr}
			V_{\infty}:=-\frac{1}{z}\cdot\frac{1}{1-zm}H_{\infty}\left(\left(\frac{1-zm}{z}\right)\right).
			\end{equation}
			
			Pour finir, montrons que, pour tout $j\in\{1,\ldots,r-1\}$, la matrice $(z-\tau_{j})V_{j}$ est non singuli\`ere dans le disque $D_{\tau_j}$. En effet, d'après \eqref{vj}, $(z-\tau_{j})V_{j}=\frac{\tau_j}{z}H_{\gamma_j}\left(\frac{1-zm}{z}\right)$. Soit $y\in D_{\tau_{j}}$, alors $\frac{1-ym}{y}\in D_{\gamma_j}$ et par conséquent, la matrice $H_{\gamma_j}\left(\frac{1-zm}{z}\right)$ n'a pas de pôle en $y$ car $H_{\gamma_j}$ appartient à $M(E(D_{\gamma_{j}}))$. Donc la matrice $(z-\tau_{j})V_{j}$ est non singuli\`ere dans le disque $D_{\tau_j}$. De mani\`ere similaire, on montre que la matrice $zV_{\infty}$ est non singuli\`ere dans le disque $D_0$. 
	
	Par cons\'equent, le théorème 4.2 de \cite{Gillesfacteurs} montre que $V$ est $E_{p}$-\'equivalente \`a la matrice
	\begin{equation}\label{F}
	F=\sum_{i=1}^{r}\frac{1}{z-\tau_{j}}G_{j},
	\end{equation}
	o\`u, pour tout $j\in\{1,\ldots, r-1\}$, $G_{j}$ est une matrice carr\'ee \`a coefficients dans $\mathbb{C}_{p}$ semblable \`a $[(z-\tau_{j})V_{j})](\tau_{j})$, $G_{r}$ est semblable \`a $(zV_{r})(0)$ et $\sum_{i=1}^r G_{j}$ est une matrice diagonale \`a coefficients dans $\mathbb{Z}$. 
	De \eqref{vj} et \eqref{vr} il suit que
	\begin{equation}\label{semblable}
	[(z-\tau_{j})V_{j})](\tau_{j})=H_{\gamma_j}(\gamma_{j})\quad\textup{et}\quad (zV_{\infty})(0)=-H_{\infty}(\infty).
	\end{equation}  
	En particulier, le théorème 4.2 de \cite{Gillesfacteurs} montre qu'il existe $T\in {\rm GL}_n(E_{p})$ telle que \[\frac{d}{dz}T=VT-TF.\] 
	En posant $H_{2}=T\left(\frac{1}{z+m}\right)$, on a $H_{2}\in {\rm GL}_n(E_{p})$ et 
	\begin{equation}\label{f}
	\begin{split}
	\frac{d}{dz}H_{2}&=\left[V\left(\frac{1}{z+m}\right)T\left(\frac{1}{z+m}\right)-T\left(\frac{1}{z+m}\right)F\left(\frac{1}{z+m}\right)\right]\left(\frac{-1}{(z+m)^2}\right)\\
	& =\left[-(z+m)^2BH_2-H_2F\left(\frac{1}{z+m}\right)\right] \left(\frac{-1}{(z+m)^2}\right)\\
	& =BH_2-H_2\cdot\left(\frac{-1}{(z+m)^2}F\left(\frac{1}{z+m}\right)\right).
	\end{split}
	\end{equation}
	Ainsi, $B$ est $E_p$-\'equivalente \`a $G:=-\frac{1}{(z+m)^{2}}F(\frac{1}{z+m})$. Comme $\tau_{j}=1/(\beta_{\gamma_{i}}+m)$, l'\'egalit\'e \eqref{F} donne que
	$$G=\frac{\beta_{\gamma_{1}}+m}{(z+m)(z-\beta_{\gamma_{1}})}G_{1}+\cdots+\frac{\beta_{\gamma_{r-1}}+m}{(z+m)(z-\beta_{\gamma_{r-1}})}G_{r-1}-\frac{1}{z+m}G_{r}.$$ 
	Finalement, d'après \eqref{semblable}, $G_j$ est semblable à $H_{\gamma_j}(\gamma_{j})$ mais, $H_{\gamma_j}(\gamma_{j})=[(z-\beta_{\gamma_j})B_{\gamma_j}](\beta_{\gamma_j})$ donc, $G_j$ est semblable à $[(z-\beta_{\gamma_j})B_{\gamma_j}](\beta_{\gamma_j})$. Encore par \eqref{semblable}, la matrice $G_r$ est semblable à $-H_{\infty}(\infty)$, mais $H_{\infty}(\infty)=[zB_{\infty}](\infty)$ alors $-H_{\infty}(\infty)$ est est semblable $[zB_{\infty}](\infty)$. 
 \end{proof}

\begin{remarque}\label{apparant}
	Il suit de la démonstration du théorème 4.2 de \cite{Gillesfacteurs} que, l'infini est une singularité régulière apparente de $F$. Alors, $-m$ est une singularité régulière apparente de $G$ et par conséquente, $G$ a une base de solutions à coefficients dans $\mathbb{C}_p((z+m))$.
\end{remarque}


\subsection{Lemmes pr\'eparatoires.}

Afin de d\'emontrer les trois pas pr\'ec\'edents, nous aurons besoin des lemmes pr\'eparatoires suivants. 

Le premier lemme est un r\'esultat qui semble classique mais nous utiliserons les id\'ees de la preuve dans la d\'emonstration du th\'eor\`eme~\ref{rig}. On pose \begin{equation}\label{l}
L=\frac{d^{n}}{dz^{n}}+a_{1}(z)\frac{d^{n-1}}{dz^{n-1}}+\cdots+a_{n-1}(z)\frac{d}{dz}+a_{n}(z)\in\mathbb{Q}(z)[d/dz],
\end{equation}

\begin{lem}\label{exp}
	Consid\'erons l'op\'erateur diff\'erentiel \eqref{l} et soit $A$ la matrice compagnon associ\'ee \`a $L$.  Si $\gamma\in\overline{\mathbb{Q}}\cup\{\infty\}$ est un point singulier r\'egulier de $L$, alors il existe $A_{\gamma}\in {\mathcal M}_n(\overline{\mathbb{Q}}(z))$ telle que:
	\begin{enumerate}
		\item  Les matrices $A$ et $\frac{1}{z-\gamma}A_{\gamma}$ sont $\mathbb{Q}(\gamma)(z)$-\'equivalentes. Et si $\gamma$ est l'infini alors $A$ et $\frac{-1}{z}A_{\infty}$ sont $\mathbb{Q}(z)$-\'equivalentes.
		\item $A_{\gamma}$ n'a pas de p\^ole en $\gamma$;
		\item Les valeurs propres de $A_{\gamma}(\gamma)$ sont les exposants de l'\'equation en $\gamma$. 
	\end{enumerate} 
\end{lem} 

\begin{remarque}\label{L et A}
	Ce lemme est encore valide quand $L\in\mathbb{C}(z)[d/dz]$, mais dans ce cas nous pouvons seulement affirmer que $A$ et $\frac{1}{z-\gamma}A_{\gamma}$ sont $\mathbb{C}(z)$-équivalentes et il en va de même pour l'infini. En particulier si $L$ est fuchsien alors le système différentiel $\frac{d}{dz}y=Ay$ est fuchsien.
\end{remarque}

\begin{proof}
	Notons que $z^{n}\frac{d^{n}}{dz^{n}}=\delta(\delta-1)\cdots(\delta+n-1)$, où $\delta=z\frac{d}{dz}$. Ainsi, il existe 
	$a_{1j},\ldots,a_{jj}\in\mathbb{Z}$ tels que $\frac{d^{j}}{dz^{j}}=\frac{a_{1j}}{z^{j}}\delta+\cdots+\frac{a_{jj}}{z^{j}}\delta^{j}$. 
	Soit 
	\[
	G_{n}=\begin{pmatrix}
	1 & 0 & 0 & 0 & \dots & 0 & 0\\
	0 & \frac{1}{z} & 0 & 0 & \dots & 0 & 0\\
	0 & -\frac{1}{z^{2}} & \frac{1}{z^{2}} & 0 & \dots & 0 & 0\\ 
	\vdots & \vdots & \vdots & \vdots & \dots  & \vdots & \vdots\\
	
	0 & \frac{a_{1,n-1}}{z^{n-1}} &  \frac{a_{2,n-1}}{z^{n-1}} & \frac{a_{3,n-1}}{z^{n-1}} &\dots & \frac{a_{n-2,n-1}}{z^{n-1}} & \frac{1}{z^{n-1}}
	\end{pmatrix}\in {\rm GL}_n(\mathbb{Z}[1/z])
	\]
	la matrice exprimant $\{1,\dots,\frac{d^{n-1}}{dz^{n-1}}\}$ en fonction de $\{1,\delta,\ldots,\delta^{n-1}\}$.
	
	Soit  $\gamma$ un point singulier r\'egulier de $L$. Nous commen\c{c}ons par consid\'erer un nouvel op\'erateur 
	$$
	L_{\gamma}:=\frac{d^{n}}{dz^{n}}+a_{1}(z+\gamma)\frac{d^{n-1}}{dz^{n-1}}+\cdots+a_{n-1}(z+\gamma)\frac{d}{dz}y+a_{n}(z+\gamma)$$
	que l'on \'ecrit sous la forme 
	\begin{equation}\label{delta}
	\delta^{n}+q_{1}(z)\delta^{n-1}+\cdots+q_{n-1}(z)\delta +q_{n}(z),
	\end{equation}
	o\`u  $\delta=z\frac{d}{dz}$  et
	\begin{equation}\label{q}
	(q_n(z),\ldots,q_1(z),1)=(a_{n}(z+\gamma),\ldots,a_{1}(z+\gamma),1)z^{n}G_{n+1}.
	\end{equation}
	On rappelle que les exposants de \eqref{l} en $\gamma$ sont les z\'eros du polyn\^ome $$\lambda^n+q_1(0)\lambda^{n-1}+\cdots+q_{n-1}(0)\lambda+q_n(0),$$ o\`u $q_{n}(z),\ldots, q_{1}(z)$ sont d\'efinis dans \eqref{q}.
	
	\smallskip
	
	Notons que $q_{i}(z)\in\mathbb{Q}(\gamma)(z)$ pour $1\leq i\leq n$. 	  
	Soit 
	\[
	B_{\gamma}=\begin{pmatrix}
	0 & 1 & \dots & 0 & 0\\
	\vdots & \vdots & \vdots & \vdots & \vdots\\
	0 & 0 & \dots & 0 & 1\\
	-q_{n}(z) & -q_{n-1}(z) & \dots & -q_{2}(z) & -q_{1}(z)
	\end{pmatrix}.
	\]
	D'apr\`es le th\'eor\`eme de Fuchs, $B_{\gamma}$ n'a pas de p\^ole en z\'ero. Montrons que les valeurs propres de $B_{\gamma}(0)$ sont les exposants en $\gamma$. En effet, le polyn\^ome caract\'eristique de $B_{\gamma}(0)$ est $(-1)^{n}(\lambda^n+q_1(0)\lambda^{n-1}+\cdots+q_{n-1}\lambda+q_n(0))$. Posons 
	\begin{equation}\label{aj}
	A_{\gamma}=B_{\gamma}(z-\gamma)\in {\mathcal M}_{n}(\mathbb{Q}(\gamma)(z)).
	\end{equation}
	Alors, $A_{\gamma}$ n'a pas de p\^ole en $\gamma$ et les valeurs propres de $A_{\gamma}(\gamma)$ sont les exposants en $\gamma$. Ainsi, elle appartient \`a ${\mathcal M}_n(\mathbb{Q}\{z-\gamma\})$. Donc $A_{\gamma}$ satisfait aux points 2 et 3 du lemme~\ref{exp}. Montrons qu'elle satisfait aussi le point 1.
	
	Soit $C_{\gamma}$ la matrice compagnon associ\'ee \`a l'\'equation $L_{\gamma}$. Montrons que 
	$$\frac{d}{dz}G_{n}=C_{\gamma}G_{n}-G_{n}\frac{1}{z}B_{\gamma}.$$ Soit $f$ une solution quelconque de $L_{\gamma}$. On pose $(\textbf{d}f)=(f,f',\ldots,f^{(n-1)})^t$ et $(\boldsymbol{\delta}f)=(f,\delta f,\ldots,\delta^{(n-1)}f)^t$. Alors, $\frac{d}{dz}(\textbf{d}f)=C_{\gamma}(\textbf{d}f)$ et $\delta(\boldsymbol{\delta}f)=B_{\gamma}(\boldsymbol{\delta}f)$. D'après la construction de la matrice $G_n$, nous avons que $(\textbf{d}f)=G_n(\boldsymbol{\delta}f)$. En appliquant $d/dz$ à cette dernière égalité, on a que $$C_{\gamma}(\textbf{d}f)=\left(\frac{d}{dz}G_{n}\right)(\boldsymbol{\delta}f)+G_n\frac{d}{dz}(\boldsymbol{\delta}f)=\left(\frac{d}{dz}G_{n}\right)(\boldsymbol{\delta}f)+G_n\frac{1}{z}\delta(\boldsymbol{\delta}f).$$ Mais $(\textbf{d}f)=G_n(\boldsymbol{\delta}f)$ et $\delta(\boldsymbol{\delta}f)=B_{\gamma}(\boldsymbol{\delta}f)$. Donc,  $$C_{\gamma}G_{n}(\boldsymbol{\delta}f)=\left(\frac{d}{dz}G_{n}\right)(\boldsymbol{\delta}f)+G_{n}\frac{1}{z}B_{\gamma}(\boldsymbol{\delta}f).$$
	Finalement, comme $f$ est une solution quelconque de $L_{\gamma}$ et que l'espace des solutions de $L_{\gamma}$ est de dimension $n$ sur $\mathbb{C}$, on obtient que  $$C_{\gamma}G_{n}=\frac{d}{dz}G_{n}+G_{n}\frac{1}{z}B_{\gamma}.$$
	Notons que $A=C_{\gamma}(z-\gamma)$, alors en posant $H_{\gamma}:=G_{n}(z-\gamma)\in {\rm GL}_n(\mathbb{Q}(\gamma)(z))$, il vient \[\frac{d}{dz}H_{\gamma}=AH_{\gamma}-H_{\gamma}\frac{1}{z-\gamma}A_{\gamma}.\]
	Pour l'infini, consid\'erons l'\'equation diff\'erentielle \begin{equation}\label{infini}
	\frac{d^{n}}{dz^{n}}y+b_{1}(z)\frac{d^{n-1}}{dz^{n-1}}y+\cdots+b_{n-1}(z)\frac{d}{dz}y+b_{n}(z)y=0,
	\end{equation} 
	o\`u $(b_{n}(z),\ldots,b_{1}(z),1)=(a_{n}(1/z),\ldots,a_{1}(1/z),1)\frac{(-1)^{n}}{z^{2n}}W_{n}$ et \[
	W_{n}=\begin{pmatrix}
	1 & 0 & 0 & 0 & \dots & 0 & 0\\
	0 & -z^{2} & 0 & 0 & \dots & 0 & 0\\
	0 & 2z^{3} & z^{4} & 0 & \dots & 0 & 0\\ 
	\vdots & \vdots & \vdots & \vdots & \dots  & \vdots & \vdots\\
	
	0 & * &  * & * &\dots & * & (-z^{2})^{n}
	\end{pmatrix}
	\]
	qui est la matrice qui exprime le vecteur $(1,D,\ldots, D^n)$ en fonction de $(1,\frac{d}{dz},\ldots,\frac{d}{dz^n})$, o\`u $D=-z^{2}\frac{d}{dz}$. 
	
	Comme l'infini est singulier r\'egulier, alors z\'ero est un point singulier r\'egulier de \eqref{infini} et par d\'efinition les exposants \`a l'infini de $\eqref{l}$ sont les exposants en z\'ero de \eqref{infini}. Notons $\widetilde{A}$ la matrice compagnon de \eqref{infini}. Il existe une matrice $\widetilde{A}_{0}$ qui n'a pas de p\^ole en z\'ero, les valeurs propres de $\widetilde{A}_{0}(0)$ sont les exposant en z\'ero de \eqref{infini} et $\frac{1}{z}\widetilde{A}_{0}$ et $\widetilde{A}$ sont $\mathbb{Q}(z)$-\'equivalentes. Alors il existe $\widetilde{H}_{0}\in {\rm GL}_n(\mathbb{Q}(z))$ telle que 
	$$\frac{d}{dz}\widetilde{H}_{0}=\frac{1}{z}\widetilde{A}_{0}\widetilde{H}_{0}-\widetilde{H}_{0}\widetilde{A}.
	$$
	Ainsi, si $\widetilde{G}=\widetilde{H}_{0}(1/z)\in\mathbb{Q}(z)$ on a $$\frac{d}{dz}\widetilde{G}=-\frac{1}{z}\widetilde{A}_{0}(1/z)\widetilde{G}-\widetilde{G}\cdot\left(\frac{-1}{z^2}\widetilde{A}(1/z)\right).$$
	 Soit $f$ un solution de $L$ et soit $g=f(1/z)$, on pose $(\textbf{d}f)=(f,f',\ldots, f^{(n-1)})^t$, $(\textbf{d}g)=(g,g',\ldots, g^{(n-1)})^t$. À l'aide des égalités $\textbf{d}f=(W_n\textbf{d}g)(1/z)$, $\frac{d}{dz}(\textbf{d}f)=A\textbf{d}f$ et $\frac{d}{dz}(\textbf{d}g)=\widetilde{A}\textbf{d}g$, on obtient que  
	 $$\frac{d}{dz}W_{n}=-(\frac{1}{z^{2}}A(1/z))W_{n}-W_n\widetilde{A}.$$ Posons $T_{\infty}=W_{n}(1/z)\in {\rm GL}_n(\mathbb{Q}(z))$, il vient  donc 
	 $\frac{d}{dz}T_{\infty}=A(z)T_{\infty}-T_{\infty}\left(\frac{-1}{z^2}\widetilde{A}(1/z)\right).$
	Maintenant, en posant $H_{\infty}=\widetilde{G}T_{\infty}^{-1}\in {\rm GL}_n(\mathbb{Q}(z))$, il vient 
	\[\frac{d}{dz}H_{\infty}=-\frac{1}{z}\widetilde{A}_{0}(1/z)H_{\infty}-H_{\infty}A.\] Par conséquent, les matrices $A$ et $-\frac{1}{z}\widetilde{A}_{0}(1/z)$ sont $\mathbb{Q}(z)$-\'equivalentes. On pose $A_{\infty}=\widetilde{A}_{0}(1/z)$, ainsi la matrice $A_{\infty}$ satisfait aux conditions demand\'ees.
\end{proof} 

Notons que pour tout nombre premier $p$, $\gamma$ appartient \`a la cl\^oture alg\'ebrique de $\mathbb{Q}_{p}$, car il annule un polyn\^ome \`a coefficients dans $\mathbb{Q}\subset\mathbb{Q}_{p}$. La d\'emonstration pr\'ec\'edente montre que $A_{\gamma},H_{\gamma}\in {\rm GL}_n(\mathbb{Q}(\gamma)(z))$ et $A_{\infty},H_{\infty}\in {\rm GL}_n(\mathbb{Q})(z)$, notamment $A_{\gamma},H_{\gamma}\in {\rm GL}_n(E_{p})$ et ainsi les matrices $A$, $\frac{1}{z-\gamma}A_{\gamma}$ sont $E_p$-équivalentes.
On obtient le corollaire suivant. 
\begin{coro}\label{corolemme}
	Soient $L\in\mathbb{Q}(z)[d/dz]$, $A$, $A_{\gamma}\in {\mathcal M}_n(\overline{\mathbb{Q}}(z))$, $\gamma\in\overline{\mathbb{Q}}\in\{\infty\}$ comme dans le lemme~\ref{exp}.
	\begin{enumerate}
		\item Soit $p$ un nombre premier tel que le disque $D_{\gamma}\subset\mathbb{C}_{p}$ ne contient pas d'autres points singulier de $L$. Alors le disque $D_{\gamma}$ est non singulier pour $A_{\gamma}$.
		\item Soit $p$ un nombre premier tel que tous les points singuliers de $L$ \`a distance finie ont norme $p$-adique \'egale \`a 1. Supposons que l'infini est un point singulier r\'egulier de $L$. Alors le disque $D_{\infty}$ est non singulier pour la matrice $A_{\infty}$.
		\item Soit $p$ un nombre premier tel que $||A||_{p}\leq 1$ et $|\gamma|_{p}=1$, alors $||A_{\gamma}||_{p,}\leq 1$ et $||\frac{1}{z-\gamma}A_{\gamma}||_{\mathcal G,p}\leq 1$.
	\end{enumerate}
	
\end{coro}
Pour $A=(a_{ij})\in {\mathcal M}_n(E_{p})$, $||A||_{\mathcal{G},p}=Max(|a_{ij}|_{\mathcal G,p})$.
\begin{proof}
	Avec les notations de la preuve du lemme~\ref{exp}, l'\'equation~\eqref{aj} donne que \[
	A_{\gamma}=\begin{pmatrix}
	0 & 1 & \dots & 0 & 0\\
	\vdots & \vdots & \vdots & \vdots & \vdots\\
	0 & 0 & \dots & 0 & 1\\
	-q_{n}(z-\gamma) & -q_{n-1}(z-\gamma) & \dots & -q_{2}(z-\gamma) & -q_{1}(z-\gamma)
	\end{pmatrix}.
	\]
	et de \eqref{q} on a que
	\begin{equation}\label{q'}
	(q_n(z-\gamma),\ldots,q_1(z-\gamma),1)=
	(a_{n}(z),\ldots,a_{1}(z),1)(z-\gamma)^{n}G_{n+1}(z-\gamma).
	\end{equation}
	Notons que $(z-\gamma)^nG_{n+1}(z-\gamma)\in {\rm GL}_n(\mathbb{Z}[z-\gamma])$.
	
	\begin{enumerate}
		\item Cela revient \`a montrer que $A_{\gamma}\in {\mathcal M}_n(E(D_{\gamma}))$, c'est-\`a-dire que $A_{\gamma}$ n'a pas de p\^ole en $D_{\gamma}$. En effet, soit $\alpha\in\overline{\mathbb{Q}}$ une singularit\'e de $A_{\gamma}$. D'apr\`es le lemme~\ref{exp}, on a $\alpha\neq\gamma$ car $A_{\gamma}$ n'a pas de p\^ole en $\gamma$. L'\'equation \eqref{q'} donne que $\alpha$ est une singularit\'e de $A$ et d'apr\`es l'hypoth\`ese, $\alpha\notin D_{\gamma}$, ainsi $A_{\gamma}\in {\mathcal M}_n(D_{\gamma})$.
		
		\item Rappelons que $A_{\infty}=\widetilde{A}_{0}(1/z)$. Notons que les singularit\'es de \eqref{infini} sont $0$ et $1/\gamma$, avec $\gamma$ une singularit\'e de $L$. De l'hypoth\`ese, $|1/\gamma|_{p}=1$, ainsi $D_{0}$ ne contient pas d'autre point singulier de \eqref{infini}. Alors, du point 1 du corollaire, $A_{0}$ est non singulier dans le disque $D_{0}$ et ainsi $A_{\infty}$ est non singulier dans le disque $D_{\infty}$.
		
		\item  D'apr\`es l'hypoth\`ese, $||(a_{n}(z),\ldots,a_{1}(z),1)||_{\mathcal G,p}\leq1$, $||(z-\gamma)||_{\mathcal G,p}=1$ et notons que \[
		G_{n+1}(z-\gamma)=\begin{pmatrix}
		1 & 0 & 0 & 0 & \dots & 0 & 0\\
		0 & \frac{1}{z-\gamma} & 0 & 0 & \dots & 0 & 0\\
		0 & -\frac{1}{(z-\gamma)^{2}} & \frac{1}{(z-\gamma)^{2}} & 0 & \dots & 0 & 0\\ 
		\vdots & \vdots & \vdots & \vdots & \dots  & \vdots & \vdots\\
		
		0 & \frac{a_{1,n}}{(z-\gamma)^{n}} &  \frac{a_{2,n}}{(z-\gamma)^{n}} & \frac{a_{3,n}}{(z-\gamma)^{n}} &\dots & \frac{a_{n-1,n}}{(z-\gamma)^{n}} & \frac{1}{(z-\gamma)^{n}}
		\end{pmatrix}.
		\] 
		Comme les $a_{ij}$ sont entiers, on a $||G_{n+1}(z-\gamma)||_{\mathcal G,p}\leq1$ et l'\'equation \eqref{q'} donne $$||(q_n(z-\gamma),\ldots,q_1(z-\gamma),1)||_{\mathcal G,p}\leq1.$$
		Ainsi, $||A_{\gamma}||_{\mathcal G,p}\leq1$.
		Finalement, comme $||\frac{1}{z-\gamma}||_{\mathcal G,p}=1$, alors $||\frac{1}{z-\gamma}A_{\gamma}||_{\mathcal G,p}=||\frac{1}{z-\gamma}||_{\mathcal G,p}||A_{\gamma}||_{\mathcal G,p}\leq 1.$
	\end{enumerate}
	
\end{proof}
Nous aurons besoin du lemme suivant pour la démonstration du deuxième pas du théorème~\ref{rig}.

\begin{lem}\label{padique}
	Soient $a_0(z)\in\mathbb{Q}[z]$, $K$ le corps de décomposition de $a_0(z)$ sur $\mathbb{Q}$, $h_2$ la dimension de $K$ comme $\mathbb{Q}$-espace vectoriel et $s$ la valuation de $a_0(z)$. Considérons l'ensemble $\mathcal{S'}$ des nombres premiers tels que le coefficient dominant de $a_0(z)$ et le terme constant du polynôme $a_0(z)/z^s$ aient une norme $p$-adique égale à 1 et  que $a_0(z)\in\mathbb{Z}_{(p)}[z]$. Alors pour tout $p\in\mathcal{S'}$, tout entier $m\geq1$ et toute racine $\gamma$ de $a_0(z)$, on a $$\mid\gamma^{p^{mh_2}}-\gamma\mid_p<\frac{1}{p}<\mid\pi_p\mid.$$  
\end{lem}

\begin{remarque}
	Notons que l'ensemble $\mathcal{P}\setminus\mathcal{S'}$ est fini.
\end{remarque}

\begin{proof}
	Si $\gamma=0$ il n'a rien à prouver. Supposons donc que $\gamma\neq0$. Soit $p\in\mathcal{S}'$, alors $a_0(z)\in\mathbb{Z}_p[z]$ car $\mathbb{Z}_{(p)}\subset\mathbb{Z}_p$. Soient $K_p$ le corps de décomposition de $a_0(z)$ sur $\mathbb{Q}_p$ et $h_3$ la dimension de $K_p$ comme $\mathbb{Q}_p$-espace vectoriel. Il suit du théorème 6.1 de \cite[Chap. I]{Dworkgfunciones} que les racines non nulles de $a_0(z)$ ont une norme égale à 1.  Soit $P(z)=a_0(z)/z^s$. Notons que $P_0(z)\in\mathbb{Z}_{(p)}[z]$. Comme $|P(0)|_{p}=1$ et toutes les racines de $P(z)$ ont norme égale à 1 alors $P(z)$ est irréductible sur $\mathbb{Q}_p$. Soit $\bar{P}(z)$ le polynôme qui est obtenu après avoir réduit les coefficients de $P(z)$ modulo l'idéal maximal de $\mathbb{Z}_{(p)}$. Comme $P(z)$ est irréductible sur $\mathbb{Q}_p$ et $P(z)$ est à coefficients dans $\mathbb{Z}_{(p)}$ alors le lemme de Hensel nous assure que $\bar{P}(z)$ est irréductible sur $\mathbb{F}_p$. Soient $L=\mathbb{Q}_p(\gamma)$ et $k=\mathbb{F}_p(\bar{\gamma})$. Comme $P(\gamma)=0$ et $P(z)$ est irréductible sur $\mathbb{Q}_p$ alors $[L:\mathbb{Q}_p]=deg(P(z))$. De même, comme $\bar{P}(\bar{\gamma})=0$ et $\bar{P}(z)$ est irréductible sur $\mathbb{F}_p$ alors $[k:\mathbb{F}_p]=deg(\bar{P}(z))$. Mais, $deg(P(z))=deg(\bar{P}(z))$ et ainsi, $[L:\mathbb{Q}_p]=deg(P(z))=[k:\mathbb{F}_p]$. Alors, d'après le petit théorème de Fermat, $\gamma^{p^{deg(P(z))}}\equiv\gamma\mod p$. Puisque $L$ est un sous corps de $K_p$ alors $deg(P(z))$ divise $h_3$ et ainsi, $\gamma^{p^{h_3}}\equiv\gamma\mod p$. Donc, pour tout entier $n\geq1$,  $\gamma^{p^{nh_3}}\equiv\gamma\mod p$.
	Finalement, comme $K$ est une extension galoisienne de $\mathbb{Q}$, on obtient que $h_2=h_3r$, où $r$ est un entier strictement positif. Ainsi, pour tout entier $m\geq1$,  $\gamma^{p^{mh_2}}\equiv\gamma\mod p.$ Finalement, on obtient que pour tout entier $m\geq1$, \[|\gamma^{p^{mh_2}}-\gamma|_{p}<\frac{1}{p}<|\pi_{p}|.\]
\end{proof}
L'importance du troisi\`eme pas repose sur le fait que nous allons passer de la th\'eorie des \'equations diff\'erentielles p-adiques \`a la th\'eorie des \'equations diff\'erentielles classiques, et pour cela nous utiliserons le th\'eor\`eme~\ref{dmfs}. Dans un premier temps, nous \'etudions les disques singuliers r\'eguliers de $B$. D'apr\`es le lemme~\ref{exp}, on obtient que $B$ est $E_{p}$-\'equivalente \`a $\frac{pz^{p-1}}{z^{p}-\gamma_{j}}A_{\gamma_{j}}(z^{p})$. De plus, d'apr\`es le choix de $p$, le corollaire~\ref{corolemme} nous garantit que la matrice $A_{\gamma_{j}}(z^{p})$ est non singuli\`ere dans le disque $D_{\gamma_{j}}$. Ainsi, la matrice  $(z-\gamma_{j})(\frac{1}{z^{p}-\gamma_{j}})A_{j}(z^p)$  a comme singularit\'es les racines $p$-i\`eme de $\gamma_{j}$ qui sont dans $D_{\gamma_j}$. Pour surmonter cette difficult\'e, nous consid\`ererons la transformation $z\mapsto(z+\gamma_{j})^{p}-\gamma_{j}$.  

\smallskip

Dans l'\'etude de cette transformation nous devons reprendre la construction de $Frob_{z^p}$ faite dans la partie~\ref{sff}. De mani\`ere plus g\'en\'erale, pour $\omega\in E_{p}$ v\'erifiant 
\begin{equation}\label{dif}
|\omega-z^{p^{h}}|<1
\end{equation}
pour un certain $h$, nous construisons $Frob_{\omega}:E_{p}\rightarrow E_{p}$ comme suit. Nous d\'efinissons $Frob_{\omega}:\mathbb{C}_{p}(z)\rightarrow E_{p}$ donn\'e par $$Frob_{\omega}\left(\frac{\sum_{i}a_{i}z^i}{\sum_{j}b_{j}z^j}\right)=\frac{\sum_{j}Frob(a_{i})\omega^i}{\sum_{j}Frob(b_{j})\omega^{j}},$$ o\`u $(\sum_{i}a_{i}z^i)/(\sum_{j}b_{j}z^j)\in\mathbb{C}_{p}(z)$.
Remarquons que $\sum_{j}Frob(b_{j})\omega^{j}\neq0$ car d'apr\`es \eqref{dif}, $\omega$ est transcendent.
Le Frobenius $Frob_{\omega}$ est continu car c'est une isom\'etrie (voir le lemme~\ref{prop1}), il s'\'etend donc au corps des \'el\'ements analytiques $E_{p}$, not\'e encore $Frob_{\omega}:E_{p}\rightarrow E_{p}$. De plus, c'est \`a nouveau une isom\'etrie.
Soit $A$ une matrice de taille $n$ \`a coefficients $E_{p}$, on consid\`ere la matrice $F_{\omega}(A):=\frac{d}{dz}(\omega)A^{Frob_{\omega}}$, o\`u $A^{Frob_{\omega}}$ est la matrice obtenue en appliquant $Frob_{\omega}$ \`a chaque entr\'ee de $A$.   

De la proposition 4.1.2 de \cite{Gillesmoduldiff} on sait que si $A\in {\mathcal M}_{n}(E_p)$ a norme inf\'erieure ou \'egale \`a 1 alors le syst\`eme $Ay=\frac{d}{dz}y$ a une basse de solution dans l'anneau des s\'eries \`a coefficients dans $E_p$ qui convergent pour $\mid x\mid<\mid\pi_p\mid$, donc en analogie avec la proposition  4.7.3 de \cite{Gillesmoduldiff}, nous obtenons la proposition suivante.

\begin{prop}\label{foncteurisomorphes}
	Soit $\omega\in E_{p}$ tel que $|\omega-z^{p^{h}}|<|\pi_{p}|$ et  $A\in {\mathcal M}_n(E_{p})$ de norme inf\'erieure ou \'egale \`a 1. Alors $F_{\omega}(A)$ et $F_{z^{p^{h}}}(A)$ sont $E_{p}$-\'equivalentes.
\end{prop}

Nous allons utiliser la proposition~\ref{foncteurisomorphes} dans l'\'etape \textbf{(i)} du troisi\`eme pas afin de montrer que $B$ est $E_{p}$-\'equivalente \`a $F_{\omega}(\frac{1}{z-\gamma_{k}}A_{\gamma_{k}})$, avec $\omega=(z+\gamma_{j})^{p^h}-\gamma_{j}$, où $Frob(\gamma_k)=\gamma_j$. 

\medskip

Un des points importants de la d\'emonstration de la proposition 4.7.3 de \cite{Gillesmoduldiff} est que $Frob_{z^{p^{h}}}$ est une isom\'etrie. Pour appliquer la proposition~\ref{foncteurisomorphes} nous devons montrer que $Frob_{\omega}$ est une isométrie. Nous d\'emontrons cela dans le lemme suivant.
\begin{lem}\label{prop1}
	Soit $h$ un entier strictement positif et $\omega\in E_p$ tels que $|\omega-z^{p^{h}}|<1$. Alors, il existe un endomorphisme isom\'etrique, $Frob_{\omega}:E_p\rightarrow E_p$  tel que $Frob_{\omega}(z)=\omega$. De plus, pour tout $e$ dans $E_p$, on a 
	\[\frac{d}{dz}(Frob_{\omega}(e))=\frac{d}{dz}(\omega)Frob_{\omega}\left(\frac{d}{dz}e\right).\] 
\end{lem}

\begin{proof}
	Soit $Frob:\mathbb{C}_p\rightarrow\mathbb{C}_p$ l'automorphisme de Frobenius choisi dans la remarque~\ref{frobcp}. 
	On a $\overline{\omega}=\overline{z}^{p^{h}}$ et $|\omega|=1$.	
	D\'efinissons $Frob_{\omega}:\mathbb{C}_p(z)\rightarrow E_p$ de la mani\`ere suivante. 
	\'Etant donn\'es $P(z)=\sum_{i=0}^{n}a_{i}z^{i}$ et  $Q(z)=\sum_{j=0}^{m}b_{j}z^{j}\in \mathbb{C}_p[z]$, on pose 
	\[Frob_{\omega}\left(\frac{P}{Q}\right)=\frac{\sum_{i=0}^{n}Frob(a_{i})\omega^{i}}{\sum_{j=0}^{m}Frob(b_{j})\omega^{j}}.\] 
	Notons que ${\sum_{j=0}^{m}Frob(b_{j})\omega^{j}}\neq0$ car $\omega$ est transcendant sur $K$.
	Montrons que $Frob_{\omega}$ est une isom\'etrie. 
	Soient $l\in\{0,\ldots,n\}$ et $k\in\{0,\ldots,m\}$ tels que  $|a_{l}|=|P|$ et $|b_{k}|=|Q|$. 
	Alors $P=a_{l}P_{1}$ et $Q=b_{k}Q_{1}$ avec $P_{1}=\sum_{i=0}^{n}c_{i}z^{i}$ o\`u $c_{i}=\frac{a_{i}}{a_{l}}$, $Q_{1}=\sum_{j=0}^{m}d_{j}z^{j}$ et $d_{j}=\frac{b_{j}}{b_{k}}$. On pose $R=\frac{P_{1}}{Q_{1}}$. \\
	Ainsi, on a $|P_{1}|=1=|Q_{1}|$, $|R|=1$ et
	\begin{equation}\label{sigma}
	\overline{Frob_{\omega}(R)}=\left(\overline{\frac{\sum_{i}c_{i}^{p}z^{ip^{h}}}{\sum_{j} d_{j}^{p}z^{jp^{h}}}}\right)=\left(\overline{\frac{\sum_{i}c_{i}z^{ip^{h-1}}}{\sum_{j} d_{j}z^{jp^{h-1}}}}\right)^{p}=\overline{R(z^{p^{h-1}})}^{p}.
	\end{equation}

	On obtient que \[|Frob_{\omega}(R)-(R(z^{p^{h-1}}))^{p}|<1,\]  et comme $|R(z^{p^{h}-1})|=1$, alors $|Frob_{\omega}(R)|=1$ et $|Frob_{\omega}(P/Q)|=|P/Q|$.\\
	
	Donc $Frob_{\omega}:\mathbb{C}_p(z)\rightarrow E_p$ est continue et isom\'etrique. Elle peut se prolonger de mani\`ere unique au corps $E_p$ en un endomorphisme isom\'etrique. De la contruction de $Frob_{\omega}$ et de la d\'erivation de la compos\'ee, il suit que pour toute fraction rationnelle $P/Q\in \mathbb{C}_p(z)$, on a	\[\frac{d}{dz}(Frob_{\omega}(P/Q))=\frac{d}{dz}(\omega)Frob_{\omega}\left(\frac{d}{dz}(P/Q)\right).\]
	Comme $Frob_{\omega}$ et $\frac{d}{dz}$ sont continues, on obtient que, pour tout $e\in E_p$, on a  \[\frac{d}{dz}(Frob_{\omega}(e))=\frac{d}{dz}(\omega)Frob_{\omega}\left(\frac{d}{dz}e\right).\]
\end{proof}

\subsection{D\'emonstration des pas}\label{dempas}
Nous sommes maintenant pr\^ets \`a d\'emontrer le th\'eor\`eme~\ref{rig}.
\begin{proof}[\textbf{D\'emonstration du premier pas}]
	
	Soient $\gamma_{1}=0,\ldots, \gamma_{r-1},\gamma_{r}=\infty$ les points singuliers de $L$ et $A$ sa matrice compagnon. D'apr\`es le lemme~\ref{exp}, on a que $\gamma_j$ est un point singulier régulier du système $\frac{d}{dz}y=Ay$. Ainsi, d'après le lemme~\ref{monodromie}, il existe $C_j\in {\mathcal M}_{n}(\mathbb{C})$ telle que $M(A,\gamma_j)$ est conjugu\'ee \`a $\exp(2\pi i C_{j})$, où, si $\lambda,\beta$ sont deux valeurs propres différentes de $C_{j}$ alors $\lambda-\beta\notin\mathbb{Z}$ et si $\alpha_{ij}$ est un exposant de $\frac{d}{dz}y=Ay$ en $\gamma_{j}$ alors il existe un entier $m\in\mathbb{Z}$ tel que $\alpha+m$ est une valeur propre de $C$. Mais, toujours d'après le lemme~\ref{exp}, les exposants de $\frac{d}{dz}y=Ay$ en $\gamma_{j}$ sont les exposants de $L$ en $\gamma_{j}$. Ainsi, $C_j$ satisfait aux conditions \textbf{(a)} et \textbf{(b)} du premier pas	
\end{proof}

\begin{remarque}\label{a}
	D'après le lemme~\ref{exp}, nous pouvons réécrire le point \textbf{(b)} comme suit: l'ensemble des valeurs propres de $A_j(\gamma_{j})$ et l'ensemble des valeurs propres de $C_j$ sont égaux modulo $\mathbb{Z}$.
\end{remarque}

		
	 Soit $\mathfrak{A}$ l'ensemble des nombres  alg\'ebriques form\'e par les points $\gamma_{2},\ldots,\gamma_{r-1}$, les diff\'erences $\gamma_{i}-\gamma_{j}$, $i\neq j$, et les d\'enominateurs des exposants des points $\gamma_{1},\ldots,\gamma_{r}$. 
	Si $p\in\mathcal{S}$, alors les \'el\'ements de $\mathfrak{A}$ ont tous une norme 
	$p$-adique  \'egale \`a 1. En effet, comme $p\in\mathcal{S}$ et $\gamma_{2},\ldots,\gamma_{r-1}$ sont les racines non nulles de $a_0(z)$ alors,  d'après le théorème 6.1 de \cite[Chap. I]{Dworkgfunciones}, leurs normes $p$-adique sont égales à 1. 
	Le discriminant de $a_0(z)$ est donnée par $\prod_{i\neq j}(\gamma_i-\gamma_{j})^2$. Si $p\in\mathcal{S}$, 
	on a donc que $\mid\prod_{i\neq j}(\gamma_i-\gamma_{j})^2\mid_p=1$. De plus, comme la norme est ultramétrique, on obtient que $\mid\gamma_i-\gamma_{j}\mid_p\leq 1$ pour $i\neq j$. On en  d\'eduit donc que 
	$\mid\gamma_{i}-\gamma_{j}\mid_p=1$ pour $i\neq j$. Rappelons enfin que si $p\in\mathcal{S}$, alors 
	$||A||_{\mathcal G,p}\leq1$.
	
	\medskip
	
	D\'esignons par $\alpha_{1,j},\ldots,\alpha_{n,j}$ les exposants de $L$ en $\gamma_{j}$. 
	 Notons $d$  le plus petit commun multiple 
	 des d\'enominateurs des $\alpha_{i,j}$ pour $1\leq i\leq n$ et $1\leq j\leq r$.  Rappelons que $h_1=\phi(d)$. On sait que si $p\in\mathcal{S}$, on a $|d|_{p}=1$. Donc, $p$ et $d$ sont premiers entre eux et $p^{h_1}\equiv1\mod{d}$.  Ainsi, pour tout $p\in\mathcal{S}$, on obtient que
	\[p^{h_1}\equiv1\mod{d}.\]
	Finalement, pour tout $i\in\{1,\ldots,n\}$ et tout $j\in\{1,\ldots,r\}$, il vient 
	\begin{equation*}
	p^{h_1}\alpha_{i,j}\equiv\alpha_{i,j}\mod{\mathbb{Z}}.
	\end{equation*}

	D'apr\`es le lemme~\ref{padique}, on a 
	\begin{equation}\label{h}
	|\gamma_{j}^{p^{h}}-\gamma_{j}|_{p}<\mid\pi_p\mid, 
	\end{equation}
	pour tout $p\in\mathcal S$ et tout $j\in\{1,\ldots,r-1\}$. 
	Comme $h=h_1h_2$,  on obtient que 
	\begin{equation}\label{expo}
	p^{h}\alpha_{i,j}\equiv\alpha_{i,j}\mod{\mathbb{Z}},
	\end{equation} 
	pour tout $i\in\{1,\ldots,n\}$, tout $j\in\{1,\ldots,r\}$ et tout $p\in\mathcal{S}$. 
	
\begin{proof}[\textbf{D\'emonstration du deuxi\`eme pas}]	
	Soit $p\in\mathcal{S}$ et montrons que les matrices $\exp(2\pi iC_{j})$ et $\exp(2\pi ip^{h}C_{j})$ sont conjugu\'ees. Soit $M$ la forme de Jordan de $C_j$, alors $C_j=UMU^{-1}$, où $U\in {\rm GL}_n(\mathbb{C})$. Soient $\lambda_{1,j},\ldots,\lambda_{r,j}$ les valeurs propres de $C_j$. Soit $J_{\lambda_{i,j}}$ un bloc de Jordan de $M$ qui correspond à la valeur propre $\lambda_{i,j}$. Ainsi, $J_ {\lambda_{i,j}}=\lambda_{i,j}I_{n_{\lambda_{i,j}}}+N_{\lambda_{i,j}},$ où $I_{n_{\lambda_{i,j}}}$ est la matrice identité de taille $n_{\lambda_{i,j}}$ et $N_{\lambda_{i,j}}$ est une matrice carrée nilpotente de taille $n_{\lambda_{i,j}}$. D'après le premier pas, $\lambda_{i,j}=\alpha+m$, où $\alpha$ est un exposant de $L$ en $\gamma_{j}$ et $m$ un entier. Ainsi, 
	\begin{equation}\label{expJ}
	\exp(2\pi iJ_{\lambda_{i,j}})=\exp(2\pi i\alpha)\exp(2\pi iN_{\lambda_{i,j}}).
	\end{equation} Les valeurs propres de $p^hC_j$ sont $p^h\lambda_{1,j},\ldots,p^h\lambda_{r,j}$. Puisque  $p^hC_j=U(p^hM)U^{-1}$, alors $p^hM$ est constituée des blocs de Jordan de $M$ multipliés par $p^h$. Ainsi, la matrice  $p^hJ_{\lambda_{i,j}}=p^h(\lambda_{i,j}I_{n_{\lambda_{i,j}}}+N_{\lambda_{i,j}})$ est un bloc de la matrice $p^hM$.
	 Montrons que $p^h\lambda_{i,j}\equiv\alpha\mod\mathbb{Z}$. Comme on l'a déjà dit $\lambda_{i,j}=\alpha+m$, où $\alpha$ est un exposant de $L$ en $\gamma_{j}$ et $m$ un entier. Alors $p^h\lambda_{i,j}=p^h\alpha+p^hm$ et il suit de \eqref{expo} que $p^h\alpha=\alpha+m'$, où $m'$ est un entier. Donc, $p^h\lambda_{i,j}=\alpha+m'+p^hm.$ Comme $m'+p^hm$ est un entier, on a $p^h\lambda_{i,j}\equiv\alpha\mod\mathbb{Z}$. Par conséquent, on obtient que 
	 \begin{eqnarray}\label{expp^hj}
	 \exp(p^h2\pi iJ_{\lambda_{i,j}})&=&\exp(2\pi ip^{h}\lambda_{i,j})\exp(2\pi ip^hN_{\lambda_i})\\
	\nonumber &=&\exp(2\pi i\alpha)\exp(2\pi iN_{\lambda_{i,j}}).
	 \end{eqnarray}
	 D'apr\`es \eqref{expJ} et \eqref{expp^hj}, on a  $\exp(2\pi iJ_{\lambda_{i,j}})=\exp(p^h2\pi iJ_{\lambda_{i,j}})$. Il suit donc que les matrices $\exp(2\pi iM)$ et $\exp(2\pi p^hM)$ sont conjuguées. Finalement, comme  $\exp(2\pi iC_j)=Uexp(2\pi iM)U^{-1}$ et $\exp(2\pi ip^hC_j)=U\exp(2\pi ip^hM)U^{-1}$, 
	 les matrices $\exp(2\pi iC_j)$ et $\exp(2\pi ip^hC_j)$ sont conjuguées. 
\end{proof}

\noindent\textit{\textbf{D\'emonstration du troisi\`eme pas}}.---  	
Soit $p\in\mathcal{S}$ et soit $B=p^{h}z^{p^{h}-1}A(z^{p^{h}})$. Nous allons montrer que $A$ et $B$ sont $E_{p}$-\'equivalentes. Dans cette partie, nous verrons $B$ comme une matrice \`a coefficients dans $E_{p}$. Pour chaque point singulier $\gamma_{1}=0,\ldots, \gamma_{r-1},\gamma_{r}=\infty$ de $L$, nous notons $A_i$ la matrice construite dans le lemme~\ref{exp} et correspondant au point $\gamma_{i}$.
 
\begin{proof}[D\'emonstration du point \rm{(i)}]	
	Nous allons appliquer le théorème~\ref{dmfs}. Pour cela nous commençons par montrer que les seuls disques singuliers r\'eguliers de $B$ sont $D_{\gamma_{1}},\ldots, D_{\gamma_{r-1}},D_{\infty}$.  Par construction de l'ensemble $\mathfrak{A}$, on a $|\gamma_j-\gamma_k|_{p}=1$ pour $j\neq k$. Donc $\overline{\gamma_{j}}\neq\overline{\gamma_{k}}$ pour $j\neq k$ et $D_{\gamma_{j}}\cap D_{\gamma_{k}}=\emptyset$ pour $j\neq k$.
	Montrons dans un premier temps que si le disque $D_{\alpha}$ de centre $\alpha$ et de rayon 1 est singulier alors il existe $j\in\{1,\ldots,r-1\}$ tel que $D_{\alpha}=D_{\gamma_{j}}$. 
	Comme $D_{\alpha}$ est un disque singulier de $B$, il existe $t\in D_{\alpha}$ tel que $t$ est une singularit\'e de $B$, alors $t^{p^{h}}$ est une singularit\'e de $A$ et il existe $i\in\{1,\ldots,r-1\}$ tel que $t^{p^{h}}=\gamma_{i}$.  D'apr\`es l'in\'egalit\'e $\eqref{h}$, on a $\overline{\gamma_{i}}=\overline{\gamma_{i}}^{p^{h}}$. Ainsi, on obtient que \[\overline{t}^{p^{h}}=\overline{\gamma_{i}}^{p^{h}}.\]
	Il suit que $\overline{t}=\overline{\gamma_{i}}$. Ainsi $|t-\gamma_{i}|<1$ et comme la norme est ultram\'etrique, on obtient que $$|\alpha-\gamma_{i}|=|\alpha-t+t-\gamma_{i}|<1$$ et $D_{\alpha}=D_{\gamma_{i}}$. 
	
	 Montrons dans un deuxi\`eme temps que, pour chaque $j\in\{1,\ldots,r-1\}$, il existe une matrice $B_{j}$ \`a coefficients dans $E_{p}$ telle que $(z-\gamma_{j})B_{j}$ est \`a coefficients dans $E(D_{\gamma_{j}})$ et $B$ est $E_{p}$-\'equivalente \`a $B_{j}$. Cela impliquera que la matrice $B$ est singuli\`ere r\'eguli\`ere dans les disques $D_{\gamma_{j}}$. D'apr\`es la remarque qui suit la preuve du lemme~\ref{exp}, il existe $H_{j}\in {\rm GL}_n(E_{p})$ telle que \[\frac{d}{dz}H_{j}=AH_{j}-H_{j}\frac{1}{z-\gamma_{j}}A_{j}.\]
	 Posons $Q_{j}:=H_{j}(z^{p^{h}})\in {\rm GL}_n(E_p)$. On a \[\frac{d}{dz}Q_{j}=\frac{d}{dz}(H_j)(z^{p^{h}})p^hz^{p^{h}-1}\]
	 et ainsi
	 \begin{equation}\label{B}
	 \frac{d}{dz}Q_{j}=BQ_{j}-Q_{j}\left(\frac{p^{h}z^{p^{h}-1}}{z^{p^{h}}-\gamma_{j}}\right)A_{j}(z^{p^{h}}).
	 \end{equation}
	 Alors la matrice $B$ est $E_p$-équivalente à $\left(\frac{p^{h}z^{p^{h}-1}}{z^{p^{h}}-\gamma_{j}}\right)A_{j}(z^{p^{h}})$.
	 
	 Notons que $Frob(\gamma_k)=\gamma_{j}$ car $Frob:\mathbb{C}_{p}\rightarrow\mathbb{C}_{p}$ fixe $\mathbb{Q}$ et ainsi les singularit\'es de $Frob(A)=A$ sont $\{Frob(\gamma_{1}),\ldots,Frob(\gamma_{r-1}),\infty\}=\{\gamma_{1},\ldots,\gamma_{r-1},\infty\}$.
	 De plus, d'apr\`es \eqref{aj}, on a $Frob(A_k)=A_j$.
	Nous allons appliquer la proposition~\ref{foncteurisomorphes} à la matrice $\frac{1}{z-\gamma_{k}}A_k$, avec $\omega=[(z-\gamma_{j})^{p^{h}}+\gamma_{j}]$.  Montrons d'abord que $|\omega-z^{p^{h}}|_{p}<|\pi_{p}|_{p}$. Comme $\omega=(z-\gamma_{j})^{p^{h}}+\gamma_{j}$, on a  \[\omega-z^{p^{h}}=(\gamma_{j}-\gamma_{j}^{p^{h}})+\sum_{k=0}^{p^{h}-1}(-1)^{k}\binom{p^{h}}{k}z^{p^{h}-k}\gamma_{j}^k.\]
	D'apr\`es le th\'eor\`eme de Lucas, on a $\binom{p^{h}}{k}\equiv0\mod{p}$ pour $0\leq k\leq p^{h}-1$ et ainsi $|\binom{p^{h}}{k}|_{p}\leq|p|_{p}$. Or $\gamma_{j}\in\mathfrak{A}$, donc on a $|\gamma_{j}|_{p}=1$ et \[\left|\sum_{k=0}^{p^{h}-1}(-1)^{k}\binom{p^{h}}{k}z^{p^{h}-k}\gamma_{j}^{k}\right|_{p}\leq |p|_{p}<|\pi_{p}|_{p}.\]
	L'in\'egalit\'e \eqref{h} donne alors que $|\gamma_{j}^{p^{h}}-\gamma_{j}|_{p}<|\pi_{p}|_{p}.$ Il suit
	\begin{equation}\label{omega}
	|\omega-z^{p^{h}}|_{p}<|\pi_{p}|_{p}<1.
	\end{equation}
	
	Comme $p\in\mathcal{S}$, on a $||A||_{\mathcal G,p}\leq1$ et $|\gamma_{j}|_{p}=1$ donc le corollaire~\ref{corolemme} entra\^ine que $||\frac{1}{z-\gamma_{j}}A_{j}||_{\mathcal G,p}\leq1$. Comme $Frob:\mathbb{C}_{p}\rightarrow\mathbb{C}_{p}$ est un automorphisme isom\'etrique, on a $||\frac{1}{z-\gamma_{k}}A_{k}||_{\mathcal G,p}\leq1$. 
	
	Maintenant, la proposition~\ref{foncteurisomorphes} appliqu\'e \`a $\frac{1}{z-\gamma_{k}}A_{k}$ implique que les matrices 
	$$
	\frac{d}{dz}(z^{p^{h}})\left(\frac{1}{z-\gamma_{k}}A_{k}\right)^{Frob_{z^{p^{h}}}}\quad\textup{et}\quad B_j=\frac{d}{dz}(\omega)\left(\frac{1}{z-\gamma_{k}}A_{k}\right)^{Frob_{\omega}}
	$$
	sont $E_{p}$-\'equivalentes. D'apr\`es la construction de $Frob_{z^{p^{h}}}$,  on a   $$\frac{d}{dz}(z^{p^{h}})\left(\frac{1}{z-\gamma_{k}}A_{k}\right)^{Frob_{z^{p^{h}}}}=\frac{p^{h}z^{p^{h}-1}}{z^{p^{h}}-\gamma_{j}}A_{j}(z^{p^{h}}).$$
	Notons que \[\frac{d}{dz}(\omega)=p^h[(z-Frob(\gamma_j))^{p^{h}-1}],\]
	et donc
	\begin{equation}\label{bj}
	B_{j}=\frac{p^{h}}{z-\gamma_{j}}A_{j}((z-\gamma_{j})^{p^{h}}+\gamma_{j}).
	\end{equation}
	D'apr\`es \eqref{B}, il découle que la matrice $B$ et $B_j$ sont $E_{p}$-\'equivalentes.	De plus, on a $(z-\gamma_{j})B_{j}=p^{h}A_{j}((z-\gamma_{j})^{p^{h}}+\gamma_{j})$. Maintenant, si $y\in D_{\gamma_j}$ alors $(y-\gamma_{j})^{p^{h}}+\gamma_{j}\in D_{\gamma_j}$ et comme $A_{j}$ est non singuli\`ere dans le disque $D_{\gamma_{j}}$, on obtient que $(z-\gamma_{j})B_{j}\in {\mathcal M}_n(E(D_{\gamma_{j}}))$.
	
	Finalement, montrons que $D_{\infty}$ est un disque singulier r\'egulier de $B$. Pour cela, nous posons
	\begin{equation}\label{binfity}
	B_{r}:=\frac{d}{dz}(z^{p^{h}})(\frac{1}{z^{p^h}})A_{r}(z^{p^{h}})
	\end{equation}
	D'apr\`es le lemme~\ref{exp}, il existe $H_{r}\in {\rm GL}_n(\mathbb{Q}(z))$ telle que
	\[\frac{d}{dz}H_{r}=AH_{r}-H_{r}\frac{1}{z}A_{r}.\]
	Posons $Q_{r}:=H_{r}(z^{p^{h}})$, alors on a \[\frac{d}{dz}Q_{r}=BQ_{r}-Q_{r}B_{r}.\]
	Ainsi la matrice $B$ est $E_{p}$-\'equivalente \`a $B_{r}$. D'apr\`es le lemme~\ref{exp}, $A_{r}$ n'a pas de p\^ole \`a l'infini et le corollaire~\ref{corolemme} donne que $A_{r}$ n'a pas de singularit\'e dans le disque $D_{\infty}$. Ainsi le disque $D_{\infty}$ est non-singulier pour $zB_{r}$, autrement dit $zB_{r}\in {\mathcal M}_n(E(D_{\infty}))$. D'après le lemme~\ref{exp}, les matrices $A$ et $\frac{1}{z-\gamma_j}A_j$ sont $\mathbb{Q}(z)(\gamma_{j})$-équivalentes. Puisque $\frac{1}{z-\gamma_j}A_j\in {\mathcal M}_{n}(\mathbb{Q}(z)(\gamma_{j}))$, à la suit de la remarque~\ref{equivalence}, il existe $W_j\in {\mathcal M}_{n}(\mathbb{C}_p(z))$ telle que $W_j$ n'a pas de pôle en $\gamma_{j}$, $C_j=W_j(\gamma_{j})$ et $\frac{1}{z-\gamma_j}A_j$ et $\frac{1}{z-\gamma_{j}}W_j$ sont $\mathbb{C}_p(z)$-équivalentes. D'où, $B_j$ et $R_j=\frac{p^h}{z-\gamma_{j}}W_j((z-\gamma_{j})^{p^h}+\gamma_{j})$ sont $E_p$-équivalentes. Montrons que $W_j$ n'a pas de pôle dans le disque $D_{\gamma_j}$. En effet, soit $P(z)/Q(z)\in\mathbb{C}_p(z)$ une entrée de $W_j$. Écrivons $P(z)=\sum_{i=0}^ra_i(z-\gamma_j)^i$, $Q(z)=\sum_{l=0}^sb_l(z-\gamma_j)^l\in\mathbb{C}_p[z-\gamma_j]$. Soient $a\in\{a_0,\ldots,a_r\}$ tel que $|a|=\max\{|a_0|,\ldots,|a_r|\}$ et $P_1(z)=\frac{1}{a}P(z)$. Notons que $|P_1(z)|_{\mathcal{G}}=1$. Soient $b\in\{b_0,\ldots,b_s\}$ tel que $|b|=\max\{|b_0|,\ldots,|b_s|\}$ et $Q_1(z)=\frac{1}{b}Q(z)$. Notons que $|Q_1(z)|_{\mathcal{G}}=1$. Alors $\frac{P(z)}{Q(z)}=\frac{a}{b}\frac{P_1(z)}{Q_1(z)}$ et $|P_1(z)/Q_1(z)|_{\mathcal{G}}=1$. Notons que les fractions rationnelles $\frac{P(z)}{Q(z)}$ et $\frac{P_1(z)}{Q_1(z)}$ ont les mêmes pôles. Par conséquent, $\gamma_j$ n'est pas un pôle de $\frac{P_1(z)}{Q_1(z)}$. Alors, $\frac{P_1(z)}{Q_1(z)}=\sum_{n\geq0}c_n(z-\gamma_j)^n$. Puisque $|P_1(z)/Q_1(z)|_{\mathcal{G}}=1$ alors, pour tout $n\geq0$, $|c_n|\leq1$. Ainsi, la réduction de $\frac{P_1(z)}{Q_1(z)}$ est égale à $\sum_{n\geq0}\overline{c_n}(z-\overline{\gamma_{j}})^n$. D'où, $\overline{\gamma_{j}}$ n'est pas un pôle de la réduction de la fraction rationnelle $\frac{P_1(z)}{Q_1(z)}$. Par conséquent, la fraction rationnelle $\frac{P_1(z)}{Q_1(z)}$ n'a pas de pôle dans le disque $D_{\gamma_{j}}$ et ainsi, la fraction rationnelle $\frac{P_(z)}{Q_(z)}$ n'a pas de pôle dans le disque $D_{\gamma_{j}}$. Par conséquent, la matrice $W_j$ n'a pas de pôle dans le disque $D_{\gamma_j}$. Alors, $(z-\gamma_{j})R_j\in {\mathcal M}_{n}(E(D_{\gamma_{j}}))$. Ainsi, d'après le théorème~\ref{dmfs}, la matrice $B$ est $E_p$-équivalente à $$G=\frac{\gamma_{1}+m}{(z+m)(z-\gamma_{1})}F_{1}+\cdots+\frac{\gamma_{r-1}+m}{(z+m)(z-\gamma_{r-1})}F_{r-1}-\frac{1}{z+m}F_{r},$$ o\`u $F_{j}\in {\mathcal M}_n(\mathbb{C}_{p})$ pour $j\in\{1,\ldots,r\}$ telles que pour $j\in\{1,\ldots,r-1\}$, $F_j$ est semblable à $[(z-\gamma_j)R_j](\gamma_{j})$, $F_r$ est semblable à $[-zR_r](\infty)$ et $\sum F_i$ est une matrice diagonale à coefficients dans $\mathbb{Z}$. De plus, il suit que pour $j\in\{1,\ldots,r-1\}$, $[(z-\gamma_j)R_j](\gamma_{j})=p^hC_j$ et de \eqref{binfity} que $[-zR_r](\infty)=-p^hC_r$. Raison pour laquelle $F_j$ est semblable à $p^hC_j$ pour $j\in\{1,\ldots,r-1\}$ et $F_r$ semblable à $-p^hC_r$.

\end{proof}

\begin{proof}[D\'emonstration du point \rm{(ii)}]
	Rappelons que $\kappa:\mathbb{C}_{p}\rightarrow\mathbb{C}$ est un isomorphisme de corps.
	En appliquant l'isomorphisme $\kappa$, il vient
	\begin{eqnarray*}
		G^{\kappa}&=&\frac{\kappa(\gamma_{1})+m}{(z+m)(z-\kappa(\gamma_{1}))}F^{\kappa}_{1}+\cdots+\frac{\kappa(\gamma_{r-1})+m}{(z+m)(z-\kappa(\gamma_{r-1}))}F^{\kappa}_{r-1}-\frac{1}{z+m}F^{\kappa}_{r}\\
		&=&-\frac{1}{z+m}\left(\sum_{j=1}^{r} F_{j}^{\kappa}\right)+\sum_{j=1}^{r-1}\frac{1}{z-\kappa(\gamma_{j})}F_{j}^{\kappa}.
	\end{eqnarray*}  
	
	Notons que la matrice  $G^{\kappa}$ est \`a coefficients dans $\mathbb{C}(z)$ et comme $F_{j}^{\kappa}\in {\mathcal M}_n(\mathbb{C})$, la matrice $G^{\kappa}$ est fuchsienne, avec $\gamma_{0}=-m,\,\kappa(\gamma_{1}),\ldots,\kappa(\gamma_{r-1})$ et l'infini comme singularit\'es.
	
	\smallskip
	
	Montrons que la matrice de monodromie de $G^{\kappa}$ en $\gamma_{0}$ est l'identit\'e. En effet, comme $\sum_{j=0}^{r}F_{j}$ est une matrice diagonale \`a coefficients dans $\mathbb{Z}$, la matrice $[(z+m)G^{\kappa}](-m)=-\sum_{j=1}^{r} F_{j}^{\kappa}$ est \'egalement diagonale et \`a coefficients entiers,  puisque tout homomorphisme de corps de caract\'eristique z\'ero fixe $\mathbb{Z}$. Ainsi, les exposants de $G^{\kappa}$ en $\gamma_{0}$ sont des entiers. Donc, d'après le lemme~\ref{monodromie}, il existe une matrice $C_0\in {\mathcal M}_{n}(\mathbb{C})$ telle que la matrice de monodromie locale de $G^{\kappa}$ en $\gamma_{0}$ est conjugu\'ee \`a $\exp(2\pi i C_0)$ et telle que l'ensemble des valeurs propres de $C_0$ est r\'eduit \`a un \'el\'ement, $\{s\}$, où $s$ est entier. Écrivons $C_0=P(D+N)P^{-1}$, où $P\in {\rm GL}_n(\mathbb{C})$, $D$ est diagonale et $N$ est  nilpotente. Le théorème 8.6 de \cite[Chap. III]{Dworkgfunciones} nous assure que $TP(X^{D}X^{N})P^{-1}$ est une solution de $G^{\kappa}$, où $T\in {\rm GL}_n(\mathbb{C}((z+m)))$. Mais, d'après la remarque~\ref{apparant}, il suit que $G^{\kappa}$ a une basse de solutions à coefficients dans $\mathbb{C}((z+m))$. Donc, $X^{N}=Id_n$. Ainsi, $C_0=PDP^{-1}$, $D=sId_n$ et $\exp(2\pi iC_0)$ est conjuguée à $Id_n$. Par conséquent, la matrice de monodromie locale de $G^{\kappa}$ en $\gamma_{0}$ est conjuguée à $Id_n$. Cela montre  que la matrice de monodromie locale de $G^{\kappa}$ en $\gamma_{0}$ est l'identité.

\vspace*{0.2cm}

	Notons que la matrice $(z-\kappa(\gamma_{j}))G^{\kappa}$ n'a pas de pôle en $\kappa(\gamma_{j})$ et que $[(z-\kappa(\gamma_{j}))G^{\kappa}](\gamma_{j})=F^{\kappa}_j$. 
 Montrons que les valeurs propres de $F_{j}^{\kappa}$ sont les valeurs propres de $C_j$ multipli\'ees par $p^{h}$.
	En effet, comme remarqu\'e pr\'ec\'edemment, $F_{j}$ est semblable \`a $p^{h}C_j$. \'Ecrivons $C_j=UMU^{-1}$, o\`u $M$ est la forme de Jordan de $C_j$ et $U\in {\rm GL}_{n}(\mathbb{C}_{p})$. On note que $M\in {\mathcal M}_{n}(\mathbb{Q})$. En effet,  d'apr\`es le lemme~\ref{exp}, les valeurs propres de $A_j(\gamma_{j})$ sont les exposants de $L$ en $\gamma_{j}$ qui sont des nombres rationnels et donc  la remarque~\ref{a} implique que les valeurs propres de $C_j$ sont des nombres rationnels. Ainsi, $F_{j}$ est semblable \`a $p^{h}M$ et, comme $p^{h}M\in {\mathcal M}_{n}(\mathbb{Q})$, on obtient que $F_{j}^{\kappa}$ est semblable \`a $p^{h}M$. Notre affirmation en d\'ecoule.
	 
		La remarque~\ref{a} entra\^ine que les valeurs propres de $p^{h}C_{j}$ sont de la forme $p^{h}\beta + s$, o\`u $s$ est un entier et $\beta$ une valeur propre de $A_{j}(\gamma_{j})$. Montrons que deux valeurs propres distinctes de $p^{h}C_{j}$  ne diff\`erent pas d'un entier. En effet, supposons que $\lambda'-\beta'\in\mathbb{Z}$, où $\lambda'$ et $\beta'$ sont deux valeurs propres de $p^hC_j$. Alors $\lambda'=p^h\lambda$ et $\beta'=p^h\beta$, où $\lambda$ et $\beta$ sont deux valeurs propres de $C_j$. Alors, $p^h(\lambda-\beta)\in\mathbb{Z}$. D'après le premier pas, il existe des entiers $m_{\lambda}$ et  $m_{\beta}$, et $\lambda_1$ et 
		$\beta_1$ des exposants de $L$ en $\gamma_{j}$, tels que $\lambda=\lambda_1+m_{\lambda}$ 
		et $\beta=\beta_1+m_{\beta}$. Alors $p^h(\lambda_1-\beta_{1})\in\mathbb{Z}$. Il découle de \eqref{expo} que $p^h\lambda_1\equiv\lambda_1\mod\mathbb{Z}$ et $p^h\beta_1\equiv\beta_1\mod\mathbb{Z}$. Ainsi, $p^h(\lambda_1-\beta_1)\equiv\lambda_1-\beta_1\mod\mathbb{Z}$. Mais comme $p^h(\lambda_1-\beta_{1})\in\mathbb{Z}$, on obtient que  $\lambda_1-\beta_1\in\mathbb{Z}$. On en d\'eduit que $\lambda-\beta=(\lambda_1-\beta_1)-(m_{\beta}-m_{\lambda})\in\mathbb{Z}$. Par conséquent, d'après le premier pas, $\lambda=\beta$ et ainsi $\lambda'=\beta'$. Alors, on a : 

\smallskip

	\textbf{(a')} deux valeurs propres de $p^{h}C_{j}$ distinctes ne diff\`erent pas d'un entier.
	
\smallskip

	Ainsi, deux valeurs propres différentes de $F^{\kappa}_j$ ne diff\`erent pas d'un entier. D'après la proposition 3.12, le lemme 3.42 et le théorème 5.1 de \cite{Singer}, la matrice de monodromie locale de $G^{\kappa}$ en $\kappa(\gamma_{j})$, $M(G^{\kappa},\kappa(\gamma_{j}))$, est conjuguée à $\exp(2\pi iF^{\kappa}_j)$. Comme $F^{\kappa}_j$ est semblable à $p^hC_j$, on obtient que les matrices 
	$\exp(2\pi iF^{\kappa}_j)$ et $\exp(2\pi ip^hC_j)$ sont conjuguées. 
	Cela donne que $M(G^{\kappa},\kappa(\gamma_{j}))$ est conjugu\'ee \`a $\exp(2\pi ip^{h}C_{j})$. D'apr\`es le deuxi\`eme pas, $\exp(2\pi i p^{h}C_{j})$ et $\exp(2\pi i C_{j})$ sont conjugu\'ees et d'après le premier pas, les matrices $\exp(2\pi i p^{h}C_{j})$ et $M(A,\gamma_j)$ sont conjuguées.  Donc $M(G^{\kappa},\kappa(\gamma_{j}))$ est conjugu\'ee \`a $M(A,\gamma_{j})$.
	
	Notons que $\{\kappa(\gamma_{1}),\ldots,\kappa(\gamma_{r-1})\}=\{\gamma_{1},\ldots,\gamma_{r-1}\}$. 
	Comme le groupe de monodromie de $A$ est rigide, les groupes de monodromie de $A$ et $G^{\kappa}$ sont conjugu\'es.
\end{proof}

\begin{proof}[D\'emonstration du point \rm{(iii)}]
	D'apr\`es ce qui pr\'ec\`ede, il existe $U\in {\rm GL}_n(\mathbb{C})$ telle que $M(A,\gamma_{j})=UM(G^{\kappa},\kappa(\gamma_{j}))U^{-1}$ pour $1\leq j\leq r$. D'apr\`es la proposition~\ref{1-1}, il existe $H_{1}\in {\rm GL}_n(\mathbb{C}(z))$ telle que \[\frac{d}{dz}H_{1}=AH_{1}-H_{1}G^{\kappa}.\]
	On pose $H=H_{1}^{\kappa^{-1}}$, ainsi $H\in {\rm GL}_n(\mathbb{C}_{p}(z))\subset {\rm GL}_n(E_{p})$ et comme $\frac{d}{dz}$ et $\kappa^{-1}$ commutent, on a \[\frac{d}{dz}H=A^{\kappa^{-1}}H-HG.\]
	Remarquons que $A^{\kappa^{-1}}=A$, puisque $A$ est une matrice \`a coefficients dans $\mathbb{Q}(z)$. Par cons\'equent, $A$ et $G$ sont $E_{p}$-\'equivalentes. D'apr\`es le point \textbf{(ii)}, $G$ est $E_{p}$-\'equivalentes \`a $B$, alors par transitivit\'e on obtient que $A$ et $B$ sont $E_{p}$-\'equivalentes.
\end{proof}

\begin{remarque}\label{h'}
	La preuve pr\'ec\'edente montre en r\'ealit\'e que pour tout entier $h'\geq 1$ tel que pour tout $p\in\mathcal{S}$ les équations \eqref{h} et \eqref{expo} sont vérifiées, alors $L$ a une structure de Frobenius forte de période $h'$.  
\end{remarque}

\section{L'opérateur hypergéométrique généralisé.}\label{operatuerhyp}

  Dans cette partie, nous \'etudions les structures de Frobenius forte des op\'erateurs hyperg\'eom\'etriques 
  g\'en\'eralis\'es. 
  Nous d\'emontrons notamment  le théorème~\ref{alghyp}.
  
  Consid\'erons l'op\'erateur diff\'erentiel hyperg\'eom\'etrique d\'efini par 
  \begin{equation}\label{hype}
  {\mathcal H}(\underline{\alpha},\underline{\beta}):-z(\delta+\alpha_{1})\cdots(\delta+\alpha_{n})y
  +(\delta+\beta_{1}-1)\cdots(\delta+\beta_{n}-1),
  \end{equation} 
  o\`u $\alpha_{1},\ldots,\alpha_{n},\beta_{1},\ldots,\beta_{n}$ sont des nombres rationnels tels que $\alpha_{i}-\beta_{j}\notin\mathbb{Z}$ pour tout $i,j\in\{1,\ldots,n\}$. 
  Cet op\'erateur est fuchsienne et a $1$, $0$ et l'infini  comme seules singularit\'es. 
  Les exposants \`a l'infini sont $\alpha_{1},\ldots,\alpha_{n}$, les exposants en $0$ sont $1-\beta_{1},\ldots,1-\beta_{n}$, 
  et les exposants en $1$ sont $0,1,\ldots,n-2, -1+\sum(\beta_{i}-\alpha_{i})$.
  Nous rappelons la d\'efinition suivante (voir \cite{Beukers}).
  \begin{defi}
  	Supposons que $a_{1},\ldots,a_{n},b_{1},\ldots,b_{n}\in\mathbb{C}^{*}$ v\'erifient $a_{i}\neq b_{j}$ pour tout $i,j\in\{1,\ldots,n\}$. 
  	Un groupe hyperg\'eom\'etrique associ\'e aux param\`etres $a_{1},\ldots,a_{n},b_{1},\ldots,b_{n}$ est un sous-groupe de ${\rm GL}_n(\mathbb{C})$ engendr\'e par des matrices $h_{0},h_{1},h_{\infty}\in {\rm GL}_n(\mathbb{C})$ telles que  
  	$h_{1}$ est une r\'eflexion et 
  	\begin{align*}
  	h_{\infty}h_{1}h_{0}&=Id, \\
  	det(z-h_{\infty})&=\prod(z-a_{i}), \\
  	det(z-h_{0}^{-1})&=\prod(z-b_{j}).
  	\end{align*} 	
  \end{defi}
  
  D'apr\`es un r\'esultat de Levelt (voir le  th\'eor\`eme 3.5 de \cite{Beukers}),  un 
  groupe hyperg\'eometrique  tel que $a_{i}\neq b_{j}$ pour tout $i,j\in\{1,\ldots,n\}$ est rigide. 
  D'autre part, il est connu que le groupe de monodromie de l'op\'erateur hyperg\'eom\'etrique 
  ${\mathcal H}(\underline{\alpha},\underline{\beta})$ est un groupe hyperg\'eom\'etrique associ\'e aux param\`etres 
  $a_{i}=exp(2\pi i\alpha_{i})$ et $b_{i}=exp(2\pi i\beta_{i})$ (voir \cite{Beukers}). Ainsi, le groupe de monodromie de \eqref{hype} est rigide. 
  Donc, gr\^ace au th\'eor\`eme~\ref{rig},  $\mathcal{H}(\underline{\alpha},\underline{\beta})$ a une structure de Frobenius forte pour presque tout nombre premier $p$.
En appliquant le théorème~\ref{rig} on a le théorème suivant.

\begin{theo}\label{hyp}
	Soit $\mathcal{S}$  l'ensemble des nombres premiers $p$ tels que $|\alpha_{i}|_{p},|\beta_{j}|_{p}\leq1$
	pour tout $i,j\in\{1,\ldots,n\}$.
	Alors, pour tout $p\in\mathcal{S}$, l'op\'erateur hyperg\'eom\'etrique ${\mathcal H}(\underline{\alpha},\underline{\beta})$ poss\`ede  une structure de Frobenius forte de p\'eriode $h=\varphi(d_{\alpha,\beta})$, o\`u $\varphi$ est l'indicatrice d'Euler et $d_{\alpha,\beta}$ est le plus petit commun multiple 
	des d\'enominateurs de $\alpha_{1},\ldots,\alpha_{n}$, $\beta_{1},\ldots,\beta_{n}$.
\end{theo}

\begin{proof}
	 En d\'eveloppant l'\'equation $\eqref{hype}$, on obtient $$(1-z)\delta^{n}+[S_{n,1}(\underline{\beta}-1)-zS_{n,1}(\underline{\alpha})]\delta^{n-1}+\cdots+S_{n,n}(\underline{\beta}-1)-zS_{n,n}(\underline{\alpha}),$$ o\`u $\underline{\beta}-1=(\beta_{1}-1,\ldots,\beta_{n}-1)$ et $S_{n,k}=\sum_{1\leq i_{1}<\cdots<i_{k}\leq n}X_{1}\cdots X_{i_{k}}.$
	
	En \'ecrivant cette \'equation en fonction de l'op\'erateur $\frac{d}{dz}$, il vient  $$L_{\underline{\alpha},\underline{\beta}}:=a_0(z)\frac{d}{dz^{n}}y+a_{1}(z)\frac{d}{dz^{n-1}}y+\cdots+a_{n}(z).$$  Soit $A$ la matrice compagnon de ce nouvel op\'erateur. D'apr\`es l'\'equation \eqref{q}, on a 
	\begin{multline*}
	\left(S_{n,n}(\underline{\beta}-1)-zS_{n,n}(\underline{\alpha}),\ldots,S_{n,1}(\underline{\beta}-1)-zS_{n,1}(\underline{\alpha}),1-z\right)G_{n+1}^{-1}\\
	=(a_{n}(z),\ldots,a_{1}(z),a_0(z)).
	\end{multline*} 
	Ainsi $a_0(z)=(1-z)z^n$ et le discriminant de $a_0(z)$ est 1. 
	Dans ce cas, on obtient que $\mathfrak{A}_1=\{1\}$, alors que l'ensemble $\mathfrak{A}_2$ est donné par les dénominateurs des $\alpha_i$, des $1-\beta_{j}$ et le dénominateur de $-1+\sum_{i=1}^{n}(\beta_i-\alpha_i)$. Finalement $\mathfrak{A}_3=\left\{\frac{a_i(z)}{a_0(z)}\right\}_{1\leq i\leq n}$. Comme $p\in\mathcal{S}$, on obtient que pour tout élément de $\mathfrak{A}_1\cup\mathfrak{A}_2$ est de norme $p$-adique est égale à 1. 
	A pr\'esent, montrons que, pour tout $p\in\mathcal{S}$, $\left|\frac{ a_i(z)}{a_0(z)}\right|_{\mathcal{G}}\leq 1$. 
	En effet, si $p\in\mathcal{S}$, alors  
	$$\left|\left|\left(S_{n,n}(\underline{\beta}-1)-zS_{n,n}(\underline{\alpha}),\ldots,S_{n,1}(\underline{\beta}-1)-zS_{n,1}(\underline{\alpha}),1-z\right)\right|\right|_{\mathcal G,p}\leq1.
	$$ 
	On v\'erifie ais\'ement que, pour tout nombre premier $p$, $||G_{n+1}^{-1}||_{\mathcal G,p}\leq1$. En particulier, il en est de m\^eme pour $p\in\mathcal{S}$. Ainsi, la norme de Gauss du vecteur $(a_{n}(z),\ldots,a_{1}(z),a_0(z))$ est inf\'erieure ou \'egale \`a 1 pour tout $p\in\mathcal{S}$. D'autre part, pour tout $p\in\mathcal{S}$, $\vert a_0(z)\vert_{\mathcal{G}}=1$. On obtient donc que  $\left|\frac{ a_i(z)}{a_0(z)}\right|_{\mathcal{G}}\leq 1$, pour tout $p\in\mathcal{S}$.
	
	Finalement, comme $|\alpha_{i}|_{p},|\beta_{j}|_{p}\leq1$ 
	pour tout $i,j\in\{1,\ldots,n\}$, alors $p$ ne divise pas $d_{\alpha,\beta}$. Ainsi,  pour tout $i,j\in\{1,\ldots,n\}$, on a 
	\[p^{\varphi(d_{\alpha,\beta})}\equiv1\mod{d_{\alpha,\beta}}\]
	\[p^{\varphi(d_{\alpha,\beta})}\equiv1\mod{d_{\alpha,\beta}}\] 
	et donc 
	\[p^{\varphi(d_{\alpha,\beta})}\alpha_{i}\equiv\alpha_{i}\mod{\mathbb{Z}}\]
	\[p^{\varphi(d_{\alpha,\beta})}\beta_{j}\equiv\beta_{j}\mod{\mathbb{Z}}.\]	
	Il suit que $p^{\varphi(d_{\alpha,\beta})}\alpha_{i}\equiv\alpha_{i}\mod{\mathbb{Z}}$, $p^{\varphi(d_{\alpha,\beta})}(1-\beta_{j})\equiv1-\beta_{j}\mod{\mathbb{Z}}$ et $p^{\varphi(d_{\alpha,\beta})}(-1+\sum_{i=1}^{n}(\beta_i-\alpha_i))\equiv -1+\sum_{i=1}^{n}(\beta_i-\alpha_i)\mod\mathbb{Z}$. De plus, ici les zéros de $a_0(z)$ sont $\gamma_1=0$ et $\gamma_2=1$, d'où $\vert\gamma_i^{p^{\varphi(d_{\alpha,\beta})}}-\gamma_i\vert=0<\vert\pi_p\vert$ pour $i\in\{1,2\}$. Donc, d'après la remarque~\ref{h'}, pour tout $p\in\mathcal{S}$ l'opérateur différentiel $L_{\underline{\alpha},\underline{\beta}}$ poss\`ede une structure de Frobenius forte de période $\varphi(d_{\alpha,\beta})$. 
	Ainsi,  pour tout $p\in\mathcal{S}$, l'opérateur différentiel $\mathcal{H}(\underline{\alpha},\underline{\beta})$ a une structure de Frobenius forte de période $\varphi(d_{\alpha,\beta})$.
\end{proof}

Nous sommes maintenant en mesure de démontrer le théorème~\ref{alghyp}.

\begin{proof}[D\'emonstration du th\'eor\`eme~\ref{alghyp}]
	Soit $p\in\mathcal S$.  Alors $p$ ne divise pas $d_{\alpha,\beta}$, 
	de sorte que $|\alpha_i|_p,|\beta_j|_p\leq1$ pour $i,j\in\{1,\ldots,n\}.$ D'apr\`es le théorème~\ref{hyp}, on obtient  que
	\[-z(\delta+\alpha_{1})\cdots(\delta+\alpha_{n})+(\delta+\beta_{1}-1)\cdots(\delta+\beta_{n}-1)\]
	poss\`ede une structure de Frobenius forte pour $p$ de p\'eriode $h$. Par hypoth\`ese, on a \'egalement $_{n}F_{n-1}(\underline{\alpha},\underline{\beta},z)\in\mathbb{Z}_{(p)}[[z]]$. 
	Le th\'eor\`eme~\ref{algebrique} implique donc  que la r\'eduction de $_{n}F_{n-1}(\underline{\alpha},\underline{\beta},z)$ modulo $p$ est alg\'ebrique sur $\mathbb F_p(z)$ de degr\'e major\'e par $p^{n^{2}h}$. 
	Enfin, la d\'emonstration du th\'eor\`eme~\ref{hyp} montre que l'on peut prendre $h= \phi(d_{\alpha,\beta})$.
\end{proof}

Pur conclure, nous appliquons le théorème~\ref{alghyp} aux deux séries hypergéométriques $f_2(z)=~_3F_{2}\left(\frac{1}{9},\frac{4}{9},\frac{5}{9};\frac{1}{3},1,1,z\right)$ et $f_3(z)=~_2F_1\left(\frac{1}{3},\frac{1}{2};\frac{5}{12},1,z\right)$. 

\`A ce jour, on ne sait  toujours pas si $f_2(z)$ est une diagonale de fraction rationnelle et on ne peut donc pas 
appliquer les résultats de  \cite{Bordiagonal}. D'autre part, les résultats de \cite{Borisgfonctpu} ne s'appliquent pas non plus à la série $f_2(z)$ (voir \cite[Example 8.6]{Borisgfonctpu}). 
Notons que  $f_2(z)$ est globalement bornée puisqu'on v\'erifie ais\'ement que  $f_2(27^2z)\in\mathbb{Z}[[z]]$.  
Ainsi,  $f_2(z)\in\mathbb{Z}_{p}[[z]]$ pour tout $p\not=3$. 
Notons que dans ce cas, on a $d_{\alpha,\beta}=9$, où $\alpha=\left(\frac{1}{9},\frac{4}{9},\frac{5}{9}\right)$ et $\beta=\left(\frac{1}{3},1,1\right)$. Le théorème~\ref{alghyp} nous garanti donc que, pour tout nombre premier $p\neq3$, $f_{2\mid p}$ est algébrique sur $\mathbb{F}_p(z)$ de degr\'e majoré par $p^{54}$.

On sait que la  série $f_3(z)$ n'est pas la diagonale d'une fraction rationnelle car elle n'est pas globalement 
bornée (voir proposition 1 de \cite{gilleborne}). Ainsi, on ne peut pas non plus lui appliquer les r\'esultats de 
\cite{Bordiagonal}. D'autre part,  les  résultats de \cite{Borisgfonctpu} ne peuvent pas lui \^etre appliqu\'es non plus (voir 
la  section 8 de \cite{Borisgfonctpu}).   Par contre, le théorème~\ref{alghyp} s'applique.  En effet, soit 
$\mathcal{\mathcal{S}}$ l'ensemble des nombres premiers tels que $f_3(z)\in\mathbb{Z}_{p}[[z]]$. Cet ensemble est infini car si $p$ est un nombre premier congruent à 1 modulo 12, alors $p\in\mathcal{S}$. 
Dans ce cas, on v\'erifie que  $d_{\alpha,\beta}=12$, où $\alpha=\left(\frac{1}{3},\frac{1}{2}\right)$ et $\beta=\left(\frac{5}{12},1,\right)$. D'après le théorème~\ref{alghyp}, si $p>3$ appartient \`a $\mathcal{S}$, 
alors $f_{3\mid p}$ est algébrique sur $\mathbb{F}_p(z)$ de degré majoré par $p^{16}$.

\bigskip

\renewcommand{\abstractname}{Remerciements}
\begin{abstract}
L'auteur tient \`a remercier chaleureusement Gilles Christol pour ses commentaires sur une version pr\'eliminaire de cet article. Il remercie \'egalement l'arbitre pour sa lecture attentive ainsi que ses diff\'erentes remarques. 
\end{abstract}


\begin{thebibliography}{99}
	
	
	\bibitem{ADconj}
	{\sc B. Adamczewski et E. Delaygue}, communication personnelle 2019. 
	
	
	\bibitem{Bordiagonal}
	{\sc B. Adamczewski, J. P. Bell,}
	{\it Diagonalization and rationalization of algebraic Laurent series.}
	{Ann. Sci. \'Ec. Norm. Sup\'er \textbf{46} (2013), 963--1004.}
	
	\bibitem{Borisgfonctpu}
	{\sc B. Adamczewski, J. P. Bell, E. Delaygue,}
	{\it Algebraic independence of G-functions and Congruences "\`a la lucas",}
	Ann. Sci. \'Ec. Norm. Sup\'er \textbf{52} (2019), 515--559.
	
	\bibitem{Borisgfonct}
	{\sc B. Adamczewski, J. P. Bell, E. Delaygue,}
	{\it Algebraic Independence of G-functions ans Congruences "\`a la lucas",}
	prepint arXiv1603.04187, 50 pp.
	
	\bibitem{andre}
	{\sc Y. Andr\'e,}
	{G-functions and geometry,}	
	Aspects of Mathematics, E13. Friedr. Vieweg \& Sohn, Braunschweig, 1989.
	
	\bibitem {Beukers}
	{\sc F. Beukers, G. Heckman,}
	{\it Monodromy for the hypergeometric fonction $_{n}F_{n-1}$,}
	Invent. Math \textbf{95} (1989), 325--354. 
	
	
	\bibitem {Gillesfacteurs}
	{\sc G. Christol,}
	{\it Décomposition des matrices en facteurs singuliers. Applications aux équations différentielles,}
	Study Group on Ultrametric Analysis \textbf{7-8} (1979–1981), Exp. No 5, 17 pp.
	
	\bibitem {Gillesmoduldiff}
	{\sc G. Christol,}
	{Modules diff\'erentielles et \'equations diff\'erentielles $p$-adiques,}
	Queen's Papers in Pure and Applied Mathematics \textbf{66}, Queen's University, Kingston, 1983.
	
	
	\bibitem {Gillesalgebriques}
	{\sc G. Christol,}
	{\it Fonctions et \'el\'ements alg\'ebriques,}
	Pacific J. Math \textbf{125} (1986), 1-37.


	
	\bibitem {gilleborne}
	{\sc G. Christol,}
	{\it Fonctions hyperg\'eom\'etriques born\'ees,}
	Study Group on Ultrametric Analysis \textbf{14} (1986-1987), Exp. No 8, 16 pp.
	
	
	\bibitem {crew}
	{\sc R. Crew,}
	{\it Rigidity and Frobenius structure,}
	 Doc. Math \textbf{22} (2017), 287–296.
	
	
	\bibitem {delignehodge}
	{\sc P. Deligne,}
	{\it Théorie de Hodge. II,}
	 Inst. Hautes Études Sci. Publ. Math \textbf{40} (1971), 5–57.
	
	
	\bibitem {deligne}
	{\sc P. Deligne,}
	{\it Intégration sur un cycle évanescent,}
	Invent. Math \textbf{76} (1983), 129--143.

	
	\bibitem {DworksFf}
	{\sc B.M Dwork,}
	{\it On $p$-adic differential equations I. The frobenius strcuture of differential equations,}
	Bull. Soc. Math. France \textbf{39-40} (1974), 27--37.
	
	
	
	\bibitem {Dworklectures}
	{\sc B.M.\ Dwork,}
	{Lectures on $p$-adic differential equations,}
	Grundlehren der Math. Wissens-chaften \textbf{253}, Spring-Verlag, New York, 1982.

	
	\bibitem{Dworkgfunciones}
	{\sc 	B. Dwork, G. Gerotto, And F. Sullivan,}
	{An introduction to G-functions,}
	Annals of Mathematics Studies \textbf{133}, Princeton University Press, 1994.
	
	\bibitem{esnault}
	{\sc H. Esnault, M. Groechenig,}
	{\it Rigid connections and F-isocrystals,}
	prepint arXiv:1707.00752, 44 pp.
	
	\bibitem{furstenberg}
	{\sc H. Furstenberg,}
	{\it Algebraic functions over finite fields,}
	J. Algebra. \textbf{7} (1967) 271--277.
	
	
	\bibitem{localsystems}
	{\sc N. Katz,}
	{\it Rigid Local Systems,}
	Annals of Math. Studies \textbf{139}, Princeton Univ. Pres, 1996.
	
	\bibitem{Kedlaya}
	{\sc K. Kedlaya,}
	{\it $p$-adic differential equations,}
	Cambridge studies in advanced mathematics \textbf{125}, Cambridge University Press, 2010. 
	
	
	
	\bibitem{Singer}
	{\sc M. van der Put, And M. F. Singer,}
	{Galois theory of linear differential equations,}
	Grundlehren der Mathematischen Wissenschaften \textbf{328}, Springer-Verlag, Berlin, 2003.
	
	
	\bibitem {Salinier}
	{\sc A. Salinier,},
	{\it Structure de Frobenius forte de l'\'equation diff\'erentielle hyperg\'eom\'etrique,}
	C. R. Acad. Sci. Paris Sér. I Math \textbf{305} (1987), 393-396.
	
	\bibitem {Siegel}
	{\sc C. Siegel,}
	{\it Ûber einige Anwendungen diohanlischer Approximationen,}
	Abhandlungen Akad. Berlin 1929..
	

	
	
	
\end{thebibliography}
\end{document}